\documentclass[a4paper,11pt]{amsart}
\usepackage{amsmath,amssymb,url}
\usepackage[active]{srcltx}
\usepackage[T1]{fontenc}
\usepackage[latin1]{inputenc}
\usepackage{hyperref}
\usepackage{latexsym,amssymb}
\usepackage{stmaryrd}
\usepackage{xcolor}
\usepackage{accents}
\usepackage{textcomp}
\usepackage{enumerate}
\usepackage[all]{xy}
\usepackage{mathtools}
\usepackage{scalerel,stackengine}
\usepackage{comment}
\stackMath
\newcommand\reallywidehat[1]{%
\savestack{\tmpbox}{\stretchto{%
  \scaleto{%
    \scalerel*[\widthof{\ensuremath{#1}}]{\kern.1pt\mathchar"0362\kern.1pt}%
    {\rule{0ex}{\textheight}}%
  }{\textheight}%
}{2.4ex}}%
\stackon[-6.9pt]{#1}{\tmpbox}%
}

\numberwithin{equation}{section}

\newcounter{Cequ}
\newenvironment{Cequation}
  {\stepcounter{Cequ}%
    \addtocounter{equation}{-1}%
    \renewcommand\theequation{C\arabic{Cequ}}\equation}
  {\endequation}
\newcounter{Cenum}
  

\renewcommand{\ge}{\geqslant}
\renewcommand{\le}{\leqslant}
\newdir^{((}{{}*!/-5pt/\dir_{(}}
\newdir^{(((}{{}*!/-5pt/\dir^{(}}
\let\op=\llbracket
\let\cl=\rrbracket
\def\pv#1{\ensuremath{\mathsf{#1}}}

\def\Om#1#2{\ensuremath{\overline{\Omega}_{#1}{\pv{#2}}}}

\newcommand\image{\mathop{\mathrm{Im}}}
\let\cal=\mathcal
\def\Cl#1{\ensuremath{\cal{#1}}}
\newlength{\dhatheight}

\usepackage[cal=boondoxo]{mathalfa}
\newcommand\St{\mathcal{St}}
\newcommand\Null{\mathcal{N\!u\!l\!l}}

\theoremstyle{plain}
\newtheorem{Thm}{Theorem}[section]
\newtheorem{Prop}[Thm]{Proposition}
\newtheorem{Lemma}[Thm]{Lemma}

\newtheorem{Cor}[Thm]{Corollary}

\theoremstyle{definition}

\newtheorem{eg}[Thm]{Example}

\begin{document}

\title{Stone pseudovarieties}%
\thanks{The first author acknowledges partial support by CMUP, member
  of LASI, which is funded through Portuguese funds through FCT under
  the project UIDB/00144/2020. %
  The second author was supported by Grant 19-12790S of the Grant
  Agency of the Czech Republic.}

\author[J. Almeida]{Jorge Almeida}%
\address{CMUP, Dep.\ Matem\'atica, Faculdade de Ci\^encias,
  Universidade do Porto, Rua do Campo Alegre 687, 4169-007 Porto,
  Portugal}
\email{jalmeida@fc.up.pt}

\author[O. Kl\'ima]{Ond\v rej Kl\'ima}%
\address{Dept.\ of Mathematics and Statistics, Masaryk University,
  Kotl\'a\v rsk\'a 2, 611 37 Brno, Czech Republic}%
\email{klima@math.muni.cz}


\makeatletter \@namedef{subjclassname@2020}{%
  \textup{2020} Mathematics Subject Classification} \makeatother

\begin{abstract}
  Profinite algebras are the residually finite compact algebras; their
  underlying topological spaces are Stone spaces. We extend the theory
  of profinite algebras to a more general setting of Stone topological
  algebras. We introduce Stone pseudovarieties, that is, classes of
  Stone topological algebras of a fixed topological signature that are
  closed under taking Stone quotients, closed subalgebras and finite
  direct products. Looking at Stone spaces as the dual spaces of
  Boolean algebras, we find a simple characterization of when the dual
  space admits a natural structure of topological algebra. This
  provides a new approach to duality theory which
  culminates in the proof that a Stone quotient of a Stone topological
  algebra that is residually in a given Stone pseudovariety is also
  residually in it, thereby extending the corresponding result of M.
  Gehrke for the Stone pseudovariety of all finite algebras over
  discrete signatures. The residual closure of a Stone pseudovariety
  is thus a Stone pseudovariety, and these are precisely the Stone
  analogues of varieties. A Birkhoff type theorem for Stone varieties
  is also established.
\end{abstract}

\maketitle


\section{Introduction}
\label{sec:intro}

In view of Stone duality, compact 0-dimensional spaces, also known as
Stone spaces or Boolean spaces, deserve special attention.
When the dual Boolean algebras have some additional structure, this is
also carried over to the corresponding Stone spaces
\cite{Gehrke&Grigorieff&Pin:2010, Gehrke&Grigorieff&Pin:2008,
  Gehrke:2016a}.
Although not initially formulated that way, this situation may be
recognized in the seminal framework developed by Eilenberg
\cite{Eilenberg:1976} for the classification of classes of regular
languages by pseudovarieties of finite semigroups and finite monoids,
which provides a translation of combinatorial problems on languages to
algebraic problems that sometimes enables the effective solution of
the former. Eilenberg's framework has been extended in various
directions, including other algebraic structures than semigroups
\cite{Steinby:1979, Almeida:1994a, Urbat&Adamek&Chen&Milius:2017,
  Straubing&Weil:2021}.

The solution of the algebraic problems resulting from the translation
mentioned in the preceding paragraph is often expressed in terms
of verifiable pseudoidentities, which are just formal equalities between
members of a suitable relatively free Stone topological algebra. In
the classical setting, the fact that pseudovarieties are always defined
by pseudoidentities is just Reiterman's theorem
\cite{Reiterman:1982,Banaschewski:1983} and has served as motivation
to understanding the structure of relatively free profinite
semigroups.

In the classical setting of Eilenberg's theory, the Stone topological
algebras that arise are all profinite, which restricts the reach of
the theory to the realm of regular languages. Basically, because Stone
topological algebras of certain kinds, such as semigroups, are
necessarily profinite, one cannot expect to go beyond regular
languages without looking at different kinds of algebraic structures.
Progress in that direction has been previously obtained (see
\cite{Gehrke&Krebs:2017}) but the theory remains confined to word
languages.

Our aim with this paper is to develop a general study of Stone
topological algebras that may eventually lead to applications to more
general classes of languages.
We extend the notion of pseudovariety of finite algebras to Stone
topological algebras, namely by considering classes \Cl S of Stone
topological algebras that are closed under taking Stone quotients,
closed subalgebras, and finite direct products. Such classes have
associated relatively free topological algebras which play a central
role in this paper.

Viewing Stone spaces as dual spaces of Boolean algebras provides an
alternative approach to Stone topological algebras.
The possibility of defining a natural topological algebraic structure
on the dual space of a Boolean algebra turns out to admit a rather
simple characterization. This leads to an alternative approach to
duality, which may be compared with that developed in
\cite{Gehrke&Grigorieff&Pin:2008, Gehrke&Grigorieff&Pin:2010,
  Gehrke:2016a}. The focus in those papers is somewhat different from
ours: first in that bounded distributive lattices, rather than Boolean
algebras, are considered there, so that Priestley duality plays the
role of Stone duality; and second because profinite algebras are
viewed there as duals of lattices with additional unary operations.
Extra care needs to be taken in our approach due to the fact that we
are dealing with topological signatures. In the case of a Stone
signature, which is a significant generalization of a finite
signature, the duality turns out to be particularly simple. As our
main application we show that, for an arbitrary Stone pseudovariety
\Cl S, a Stone quotient of a residually \Cl S Stone topological
algebra is again residually~\Cl S. This further generalizes the
special case obtained in~\cite{Gehrke:2016a} beyond the profinite case
and discrete signatures. The residual closure of a Stone pseudovariety
is, therefore, also a Stone pseudovariety. For such Stone
pseudovarieties, which are the Stone analogs of the classical Birkhoff
varieties, we establish adequate versions of Birkhoff's theorem by
showing that they are defined by Stone pseudoidentities over Stonean
spaces. The Reiterman theorem turns out to be a special case.

Here is a short guide to the paper. Section~\ref{sec:prelims} quickly
introduces necessary background.
Section~\ref{sec:Stone-pseudovarieties} defines Stone pseudovarieties,
provides a construction of their free Stone topological algebras, and
discusses various properties of these structures and how they relate
to term algebras. In Section~\ref{sec:prel-bool-algebra}, we define
three conditions on a Boolean algebra of subsets of an algebra and,
which have both a topological and an algebraic/combinatorial
character; two of these conditions turn out to be equivalent in
general while all three conditions are equivalent for a Stone
signature. The more general of those conditions are shown in
Section~\ref{sec:duality} to characterize such Boolean algebras whose
dual spaces have a natural structure of Stone topological algebra,
thus leading to our approach to duality. Since free Stone topological
algebras are compactifications of term algebras, a natural question is
whether they are the most general ones, namely the \v Cech-Stone
compactifications; we show that this is not the case for a discrete
signature as long as there is at least one operation symbol of arity
greater than one. The paper concludes with
Section~\ref{sec:Stone-varieties}, where some first applications of
duality are presented. Further applications are planned for the
continuation of this work.

\section{Preliminaries}
\label{sec:prelims}

We adopt the notion of topological algebra given
in~\cite{Schneider&Zumbragel:2017}. 
Briefly, we consider a
\emph{signature} to be a graded set $\Omega=\bigcup_{n\ge0}\Omega_n$,
where each $\Omega_n$~is a set, whose elements are called the
\emph{$n$-ary operation symbols}. An \emph{$\Omega$-algebra} is a pair
$(A,E)$, where $A$ is a nonempty set and $E=(E_n)_{n\ge0}$ is a
sequence of \emph{evaluation mappings} $E_n:\Omega_n\times A^n\to A$.
In case $A$ is endowed with a topology and so is each set $\Omega_n$,
we say that the algebra $(A,E)$ is a \emph{topological
  $\Omega$-algebra} if each mapping $E_n$ is continuous. Usually, the
sequence $E$ is not explicitly mentioned and we refer to an algebra
$(A,E)$ simply as the algebra $A$. The notions of homomorphism,
subalgebra, direct product are the standard ones (see
\cite{Burris&Sankappanavar:1981}); in the topological setting, we want
homomorphisms to be continuous, subalgebras to have the induced
topology and product algebras to be endowed with the product topology.

Given a topological property \Cl P, we say that topological algebra
$A$ is a \emph{\Cl P algebra} if the topological space $A$ has
property \Cl P. However, we talk about a \emph{Stone topological
  algebra} instead of Stone algebra when the underlying toopological
space is a Stone space because the latter designation already has a
different meaning in the literature.

If each set $\Omega_n$ is a topological space with a certain property
\Cl P (such as being compact, 0-dimensional, or discrete), then we
say that the signature $\Omega=\bigcup_{n\ge0}\Omega_n$ has property
\Cl P. The exception is finiteness for we say that a signature is
\emph{finite} if so it is as a set. Note that, unlike some literature,
we assume that compact spaces are Hausdorff. Since we always want to
consider only Hausdorff spaces, finite algebras are always viewed with
the discrete topology.

For a class \Cl C of topological algebras, we say that a topological
algebra $A$ is \emph{residually \Cl C} if, for every pair $a,b$ of
distinct elements of~$A$, there is a continuous homomorphism
$\varphi:A\to C$ into a member $C$ of~\Cl C such that
$\varphi(a)\ne\varphi(b)$.

A continuous function from a topological space $X$ to a topological
algebra $A$ is said to be a \emph{generating mapping} if its image
generates (algebraically) a dense subalgebra of~$A$. In case the
mapping is inclusion, we then also say that $A$ is \emph{$X$-generated}.

By a \emph{profinite algebra} we mean a compact algebra which is
residually finite. Equivalently, a profinite algebra is an inverse
limit of finite algebras in the category of topological algebras. For
our purposes, the following alternative characterization of
profiniteness suffices.

\begin{Thm}[{\cite{Almeida&Klima&Goulet-Ouellet:2023}}]
  \label{t:profiniteness}
  A Stone topological algebra $S$ is profinite if and only if, for
  every clopen subset $L\subseteq S$, there exists a continuous
  homomorphism $\varphi:S\to F$ onto a finite algebra $F$ such that
  $L=\varphi^{-1}(\varphi(L))$.
\end{Thm}

In the terminology of formal language theory,
Theorem~\ref{t:profiniteness} states that a Stone topological algebra
is profinite if and only if its clopen subsets are the subsets that
are recognized by continuous homomorphisms into finite algebras.

The following result is proved in~\cite[Theorem~4.3]{Gehrke:2016a} for
finite signatures using duality theory. An alternative approach using
a characterization of profiniteness by syntactic congruences is
presented in \cite[Theorem~4.2]{Almeida&Klima&Goulet-Ouellet:2023},
where an equivalent formulation is adopted that is not so convenient
for our current purposes; it works in our setting of general
topological algebras.

\begin{Thm}
  \label{t:quotient-profinite}
  Suppose that $\varphi:S\to T$ is an onto continuous homomorphism of
  topological algebras such that $T$ is 0-dimensional and Hausdorff.
  If $S$ is profinite then so is $T$.
\end{Thm}

A \emph{pseudovariety} is a class of finite $\Omega$-algebras that is
closed under taking homomorphic images, subalgebras and finite direct
products. For a pseudovariety \pv V, a profinite algebra is said to be
\emph{pro-\pv V} if it is residually \pv V.

The absolutely free $\Omega$-algebra over a set $X$ is denoted
$T_\Omega(X)$. Its elements are the \emph{$\Omega$-terms} on~$X$, that
is, the formal expressions that may be constructed from the elements
of $X$ applying formally operation symbols according to their arity:
the elements of $X$ are terms and, if $w\in\Omega_n$ and
$t_1,\ldots,t_n$ are terms, then so is $w(t_1,\ldots,t_n)$. Terms may
be represented by labeled trees recursively as follows: the labeled
tree of $x\in X$ is a tree reduced to the root, of \emph{non-operation
  type} which is labeled $x$; for $w\in\Omega_n$ and terms
$t_1,\ldots,t_n$, the labeled tree of $w(t_1,\ldots,t_n)$ has root of
\emph{operation type} labeled $w$ with (ordered) children which are
the roots of the labeled trees of $t_1,\ldots,t_n$, respectively. By
the \emph{typed shape} of a term we mean the corresponding labeled
tree where labels are omitted but the type of node is retained.

In case $X$ is a topological space and $\Omega$ is a topological
signature, we may view the term algebra as the topological sum of the
spaces of terms of each typed shape; by the latter we mean the set of
all terms of a given typed shape, which is viewed as the direct product of
the spaces of possible labels of each individual node.

\begin{Prop}
  \label{p:term-algebra-free-as-top-algebra}
  The term algebra $T_\Omega(X)$ is the absolutely free topological
  $\Omega$-algebra on~$X$ in the sense that the inclusion mapping
  $\iota:X\to T_\Omega(X)$ is such that, for every continuous mapping
  $\varphi:X\to A$ into a topological $\Omega$-algebra $A$, there is a
  unique continuous homomorphism $\hat{\varphi}:T_\Omega(X)\to A$ such
  that $\hat{\varphi}\circ\iota=\varphi$.
\end{Prop}

\begin{proof}
  We first show that $T_\Omega(X)$ is a topological $\Omega$-algebra
  by showing that each evaluation mapping
  $E_n:\Omega_n\times\bigl(T_\Omega(X)\bigr)^n\to T_\Omega(X)$ is
  continuous. Indeed, if $(o_i,t_{1,i},\ldots,t_{n,i})_i$ is a net
  converging in $\Omega_n\times\bigl(T_\Omega(X)\bigr)^n$ to
  $(o,t_1,\ldots,t_n)$, then we may assume that, for each
  $k\in\{1,\ldots,n\}$, all $t_{k,i}$ have the same typed shape
  as~$t_k$, so that the label in each node of $t_{k,i}$ converges to
  the label of the corresponding node of ~$t_k$. Then
  the same
  property holds for the terms of the net
  $\bigl(o_i(t_{1,i},\ldots,t_{n,i})\bigr)_i$ with respect to the term
  $o(t_1,\ldots,t_n)$, which shows that this term is the limit of the
  net, thereby establishing that $E_n$ is continuous.

  Next, suppose that $\varphi:X\to A$ is a continuous mapping into a
  topological algebra $A$. By the universal property of $T_\Omega(X)$
  as a free algebra, we know that there is a unique homomorphism
  $\hat{\varphi}:T_\Omega(X)\to A$ such that
  $\hat{\varphi}\circ\iota=\varphi$. Thus, it suffices to show that
  $\hat{\varphi}$ is continuous. For this purpose, consider a net
  $(t_i)_i$ converging in~$T_\Omega(X)$ to a term $t$. We may again
  assume that all $t_i$ have the same typed shape as $t$. We need to
  show that
  \begin{equation}
    \label{eq:term-algebra-free-as-top-algebra-1}
    \text{$\hat{\varphi}(t_i)$ converges to $\hat{\varphi}(t)$.}
  \end{equation}
  We proceed by induction on the height of $t$, assuming that the
  desired convergence \eqref{eq:term-algebra-free-as-top-algebra-1}
  holds for smaller height than that of~$t$. As $t_i$ converges to~$t$
  in $T_\Omega(X)$, we may assume that either $t_i$ converges to $t$
  in $X$, in which case the continuity of $\varphi$ yields
  \eqref{eq:term-algebra-free-as-top-algebra-1}, or
  $t_i=o_i(s_{1,i},\ldots,s_{n,i})$ and $t=o(s_1,\ldots,s_n)$, where
  $n\ge0$, $o_i$ converges to~$o$ in $\Omega_n$, and each $s_{k,i}$
  converges to $s_k$ in $T_\Omega(X)$. Since each $s_k$ has height
  smaller than that of~$t$, the induction hypothesis gives that
  $\hat{\varphi}(s_{k,i})$ converges to $\hat{\varphi}(s_k)$ in~$A$.
  Since $\hat{\varphi}$ is a homomorphism and the evaluation mapping
  $E_n^A:\Omega_n\times A^n\to A$ is continuous, it follows that
  \eqref{eq:term-algebra-free-as-top-algebra-1} holds, thereby
  achieving the induction step.
\end{proof}

Let $X$ be a topological space. We say that the space $X$ is $T_1$ if,
given any distinct points $x,y\in X$, there is an open set $U$ such
that $x\in U$ and $y\notin U$. The space $X$ is \emph{completely
  regular} if, whenever $x\in X$ and $C\subseteq X$ is a closed subset
not containing $x$, there is a continuous function
$\varphi:X\to\mathbb{R}$ into the reals that maps $x$ to~$0$ and $C$
to~$1$. A $T_1$ completely regular space is said to be a
\emph{Tychonoff space}. Compact spaces are Tychonoff \cite[Theorems
1.5.11 and 3.1.9]{Engelking:1989} and so are subspaces of Tychonoff
spaces.

A \emph{compactification} of the space $X$ is a compact space $C$
endowed with a homeomorphic embedding $\varepsilon_C:X\to C$ whose
image is a dense subspace of~$C$. Compactifications of $X$ may be
naturally ordered by letting $C_1\le C_2$ if there is a continuous
mapping $f:C_2\to C_1$ such that
$f\circ\varepsilon_{C_2}=\varepsilon_{C_1}$. The \emph{\v Cech-Stone
  compactification} (also known as \emph{Stone-\v Cech}
compactification; since both papers \cite{Cech:1937,Stone:1937} were
published in the same year, we prefer the alphabetical order) of~$X$
is a maximum compactification of~$X$ in the quasi-ordering $\le$,
which may or may not exist. If it exists for the space $X$, it is
clearly unique and it is denoted $\beta X$. The \v Cech-Stone
compactification exists exactly for Tychonoff spaces
\cite[Corollary~3.5.10]{Engelking:1989}. Note that the existing
literature also considers compactifications of more general spaces,
but then the definition assumes that $\varepsilon_C$ is only a
continuous mapping.

Recall that, for a Tychonoff space $T$ of cardinality $\kappa$, any
compactification of~$T$ has cardinality at most $2^{2^\kappa}$
\cite[Theorem~3.5.3]{Engelking:1989}. Hence, for a topological space
$X$, the cardinality of an $X$-generated Stone topological
$\Omega$-algebra $S$ is bounded by $2^{2^\kappa}$, where
$\kappa=\max\{|\Omega|,\aleph_0,|X|\}$ is a bound of the cardinality
of the subalgebra of $S$ (algebraically) generated by $X$. By
identifying homeomorphic-isomorphic Stone topological algebras, we
conclude that the $X$-generated Stone topological algebras 
may be
viewed as constituting a set.

\section{Stone pseudovarieties and relatively free Stone topological algebras}
\label{sec:Stone-pseudovarieties}

By a \emph{Stone pseudovariety} we mean a nonempty class of Stone
topological algebras of a fixed topological signature that is closed
under taking continuous homomorphic images that are again Stone
spaces, closed subalgebras, and finite direct products. Clearly, every
one-point algebra, also called a trivial (topological) algebra, is a
Stone topological algebra which is a homomorphic image of every Stone
topological algebra. Thus, the trivial algebras constitute the
smallest Stone pseudovariety. We also mention that continuous
bijections between compact spaces are homeomorphisms. Hence, a
continuous mapping between Stone topological algebras that is an
algebraic isomorphism is also a homeomorphism.

Note that pseudovarieties of finite algebras are Stone
pseudovarieties. For a pseudovariety \pv V of finite algebras, the
class of all pro-\pv V algebras is also a Stone pseudovariety because
of Theorem~\ref{t:quotient-profinite}.

The class of all Stone topological $\Omega$-algebras is a Stone
pseudovariety, denoted $\St_\Omega$. The pseudovariety of all finite
Stone topological $\Omega$-algebras is denoted $\pv{Fin}_\Omega$.
Since the intersection of a nonempty family of Stone pseudovarieties
is again a Stone pseudovariety, every class of Stone topological
algebras generates a Stone pseudovariety, namely the smallest Stone
pseudovariety that contains it.

The following are some further easy to describe examples of Stone
pseudovarieties.

\begin{eg}
  \label{eg:null}
  On a Stone space, we may always define a structure of Stone
  topological algebra by choosing a point and declaring it to be the
  only value of all operations. Such a structure and the resulting
  Stone topological algebra are said to be \emph{null}. The class
  $\Null_\Omega$ of all null Stone topological $\Omega$-algebras is a
  Stone pseudovariety.
\end{eg}

\begin{eg}
  \label{eg:cardinal}
  Let $\kappa$ be an infinite cardinal. The class $\St_\Omega^\kappa$
  of all Stone topological $\Omega$-algebras of cardinal less
  than~$\kappa$ is a Stone pseudovariety. Note that
  $\St_\Omega^{\aleph_0}=\pv{Fin}_\Omega$. Distinct infinite cardinals
  $\kappa$ determine distinct Stone pseudovarieties
  $\St_\Omega^\kappa$ and in fact even the Stone pseudovarieties
  $\St_\Omega^\kappa\cap\Null_\Omega$ are distinct: first, for every
  infinite discrete space, its Alexandroff (one point)
  compactification is a Stone space, so that there are Stone spaces of
  every cardinality; second, such a Stone space admits a null
  structure.
\end{eg}

\begin{eg}
  \label{eg:nilpotent}
  A Stone topological algebra $S$ is said to be \emph{nilpotent} if
  there exists an integer $N$ such that, for every finite space $X$
  and every continuous homomorphism $\varphi:T_\Omega(X)\to S$, all
  terms with typed shape of height at least $N$ have the same image.
  The class $\Cl S_{\mathrm{nil}}$ of all nilpotent Stone topological
  algebras is a Stone pseudovariety.
\end{eg}

\subsection{Relatively free Stone topological algebras}
\label{sec:free-Stone}

Given a topological space $X$ and a Stone pseudovariety \Cl S, an
\emph{\Cl S-free Stone topological algebra over~$X$} is given by a
continuous generating mapping $\iota:X\to S$ into a Stone topological algebra $S$
which is residually \Cl S and 
such that, for every continuous mapping $\varphi:X\to T$ into a member
$T$ of~\Cl S, there is a unique continuous homomorphism
$\hat{\varphi}:S\to T$ such that $\hat{\varphi}\circ\iota=\varphi$.
Usually, the mapping $\iota$ is understood from the context and we
simply refer to $S$ as an \Cl S-free Stone topological algebra.

The proof of the next result gives a construction for free Stone
topological algebras relative to Stone pseudovarieties. The
construction is nothing but the usual realization of an inverse limit
in the topological or algebraic context. Since our considerations on
Stone topological algebras go beyond the standard setting, the details
of the proof are presented for the sake of completeness.

\begin{Prop}
  \label{p:free-Stone}
  Let \Cl S be a Stone pseudovariety and let $X$ be a topological
  space. Then there exists an \Cl S-free Stone topological algebra
  over~$X$ which is uniquely determined up to continuous isomorphism
  respecting generators.
\end{Prop}

\begin{proof}
  By the cardinality considerations at the end of
  Section~\ref{sec:prelims}, there is a nonempty set \Cl R of
  generating mappings from $X$ to members of~\Cl S such that, for
  every generating mapping $\varphi:X\to T$ into a member of~\Cl S,
  there is a unique member $\psi:X\to R$ of~\Cl R and a continuous
  isomorphism $\alpha:T\to R$ such that $\alpha\circ\varphi=\psi$. For
  a generating mapping $\psi:X\to R$ in \Cl R we also write $R_\psi$
  for $R$. We order \Cl R as follows: for $\varphi:X\to R_\varphi$ and
  $\psi:X\to R_\psi$, we write $\varphi\le\psi$ if there exists a
  continuous homomorphism $\alpha_{\varphi,\psi}:R_\psi\to R_\varphi$
  such that $\alpha_{\varphi,\psi}\circ\psi=\varphi$. We observe that
  the mapping $\alpha_{\varphi,\psi}$ is uniquely determined by the
  pair $\varphi,\psi$ of elements of~\Cl R such that $\varphi\le \psi$
  because $R_\psi$ is $X$-generated. Also note that $\le$ is a partial
  order on~\Cl R which is upper directed: given $\varphi_1$ and
  $\varphi_2$ in~\Cl R, there is $\varphi\in\Cl R$ such that
  $\varphi_1\le\varphi$ and $\varphi_2\le\varphi$.

  The remainder of the proof consists in showing that the usual
  construction of the inverse limit of the inverse system \Cl R is a
  Stone topological algebra that is \Cl S-free over~$X$.
  
  The product $P=\prod_{\varphi\in\Cl R}R_\varphi$ of Stone
  topological algebras is itself a Stone topological algebra. Consider
  the subset $F$ consisting of all $(r_\varphi)_{\varphi\in\Cl R}$
  such that, whenever $\varphi,\psi\in\Cl R$ satisfy $\varphi\le\psi$,
  the equality $\alpha_{\varphi,\psi}(r_\psi)=r_\varphi$ holds. We
  claim that the (continuous) mapping $\iota:X\to F$ given by
  $\iota(x)=(\varphi(x))_{\varphi\in\Cl R}$ determines an \Cl S-free
  Stone topological algebra over $X$.
  
  It is routine to verify
   that $F$ is a subalgebra of~$P$. We need to
  show that it is a closed subset of $P$. For a pair $\varphi,\psi\in
  \Cl R$ satisfying $\varphi\le \psi$, we consider the subset
  $F_{\varphi,\psi}=\{(r_\chi)_{\chi\in\Cl R} \in P \mid
  \alpha_{\varphi,\psi}(r_\psi)=r_\varphi\}$ of $P$. Since the graph
  of the continuous mapping $\alpha_{\varphi,\psi}$ is a closed subset
  of $R_\psi\times R_\varphi$, the subset $F_{\varphi,\psi}$ is closed
  in $P$. Now $F$ is closed because it is the intersection of all
  $F_{\varphi,\psi}$ with $\varphi,\psi\in\Cl R$ such that
  $\varphi\le\psi$. As $P$ is a Stone topological algebra that is
  residually \Cl S, the same is true for its closed subalgebra $F$.
  
  The mapping $\iota:X\to F$ is continuous. To show that it is a
  generating mapping, we recall that elements of the subalgebra
  algebraically generated by $\iota (X)$ are given by terms in
  variables from the set $X$. More formally and generally, since
  $T_\Omega(X)$ is the free topological algebra over $X$, each
  continuous mapping $\varphi: X \rightarrow R$ can be extended to a
  unique continuous homomorphism $\hat{\varphi} : T_\Omega(X)
  \rightarrow R$, where the image of $\hat{\varphi}$ is exactly the
  subalgebra of $R$ algebraically generated by $\iota(X)$. We consider
  an arbitrary element $r\in F$ and its neighborhood $N$ and we want
  to show that $N$ contains some point given by a term. We may assume
  that $N=\prod_{\psi\in\Cl R} N_\psi$ is an open set from the basis
  of the product topology of $P$. Thus, there is a finite subset $Z$
  of \Cl R such that $N_\varphi=R_\varphi$ for $\varphi\not\in Z$ and
  $N_\varphi$ is a proper open subset of $R_\varphi$ for $\varphi\in
  Z$. Since the order of \Cl R is upper directed, there is $\psi\in
  \Cl R$ such that $\varphi \le \psi$ for every $\varphi \in Z$. We
  observe that the open subset $V$ of $R_\psi$ given by
  $V=\bigcap_{\varphi\in Z} \alpha^{-1}_{\varphi, \psi}(N_\varphi)$
  contains the element $r_\psi$. Since $R_\psi$ is $X$-generated,
  there is a term $t$ such that $\hat{\psi}(t)\in V$. It is easy to
  see that the condition $\hat{\varphi}(t)\in N_\varphi$ holds for every
  $\varphi \in Z$ and so, it holds for every $\varphi \in \Cl R$.
   
  To show that $\iota:X\to F$ is an \Cl S-free Stone topological
  algebra over $X$, it remains to prove the universal property. Let
  $\varphi:X\to S$ be an arbitrary continuous mapping into a
  member $S$ of~\Cl S. 
  By the choice of~\Cl R, there exist $\psi:X\to
  R_\psi$ in~\Cl R and a continuous injective homomorphism $\alpha:R_\psi\to S$.
  Consider the restriction $\pi_\psi:F\to R_\psi$ to~$F$ of the
  projection of~$P$ to the $\psi$-component. Then
  $\alpha\circ\pi_\psi:F\to S$ is a continuous homomorphism such that
  $\alpha\circ\pi_\psi\circ\iota=\varphi$. Uniqueness of a 
  mapping $F\to S$ with such properties follows from the fact that $\iota$ is
  a generating mapping. This establishes the existence of an \Cl
  S-free topological algebra over~$X$.
  
  Let $\chi : X \rightarrow G$ be another \Cl S-free Stone topological
  algebra over X. Let $\varphi$ be an arbitrary element of \Cl R. From
  the universal property of $\chi$, we may deduce the existence of a
  continuous homomorphism $\beta_\varphi : G \rightarrow R_\varphi$
  such that $\beta_\varphi \circ \chi =\varphi$. From the universal
  property of the product (both as an algebra and a topological
  space), we get a continuous homomorphism $\beta : G \rightarrow P$.
  The image of $\beta$ is equal to $F$ because $\chi$ is a generating
  mapping. Moreover, since $G$ is residually \Cl S, $\beta$ is also
  injective. Altogether, $\beta : G \rightarrow F$ is a continuous
  isomorphism and there is just one \Cl S-free Stone topological
  algebra over $X$ up to continuous isomorphism.
\end{proof}

For a Stone pseudovariety \Cl S, we denote $\Om X{}\Cl S$ the \Cl
S-free Stone topological algebra over~$X$. The mapping $\iota:X\to\Om
X{\Cl S}$ is called the \emph{natural generating mapping}.

\begin{Prop}
  \label{p:free-Stone-embedding-of-X}
  Let \Cl S be a nontrivial Stone pseudovariety and let $X$ be a
  topological space. Then the continuous mapping $\iota:X\to\Om X{\Cl
    S}$ is a homeomorphic embedding if and only if $X$ is
  0-dimensional and Hausdorff.
\end{Prop}
\begin{proof}
  Assume that $\iota$ is a homeomorphic embedding.
  Since every Stone space is 0-dimensional and Hausdorff, and these
  properties are inherited by subspaces, we get that $X$ is
  0-dimensional and Hausdorff.
  
  Now, assume that $X$ is 0-dimensional and Hausdorff. Taking a pair
  of distinct elements $x$ and $y$ in $X$, we know that
  there is a clopen subset
  $U$ of~$X$ such that $x\in U$ and $y\not\in U$. We consider a
  function $\psi:X\to A$ into a nontrivial member $A$ of~\Cl S which
  maps $U$ to a point and $X\setminus U$ to a different point. Note
  that $\psi$ is a continuous mapping. Since $\psi$ factorizes through
  $\iota$, there is a continuous (homomorphism) $\alpha : \Om X{\Cl S}
  \to A$ such that $\alpha \circ \iota = \psi$. Since $A$ is a Stone
  space, there is is a clopen subset $V$ containing $\psi(U)$ but not
  the singleton set $\psi(X\setminus U)$. Then $\iota
  (U)\subseteq\alpha^{-1}(V)$ and $\iota (X\setminus
  U)\subseteq\iota(X)\setminus\alpha^{-1}(V)$, so that $\iota
  (U)=\alpha^{-1}(V)\cap\iota(X)$ is a clopen subset of $\iota(X)$,
  thereby showing that $\iota:X\to\iota(X)$ is an open mapping.
  In particular, we see that $\iota(x)\not=\iota(y)$, which shows that
  $\iota$ is injective. Hence, $\iota:X\to\iota(X)$ is a
  homeomorphism.
\end{proof}

Whenever the assumptions of Proposition~\ref{p:free-Stone-embedding-of-X} hold, we 
think of $X$ as a subset of~$\Om X{\Cl S}$ and of
the natural generating mapping as being the inclusion mapping.

Another natural question is whether $\Om X{\Cl S}$ contains as a
subspace a homeo\-morphic-isomorphic copy of $T_\Omega(X)$, which is
true in the case of pseudovarieties of finite algebras whenever
certain nilpotent algebras are present in \Cl S. We consider here
their topological analogs which reflect the topologies on $X$ and
$\Omega$.

\begin{Lemma}
  \label{l:nilpotent-algebra}
  Let $K$ be a clopen subset of $T_\Omega(X)$ all of whose elements
  have the same typed shape such that $K$ is the product of clopen
  sets at each node in the typed shape. Then there is a continuous
  homomorphism $\varphi:T_\Omega(X)\to F$ into a nilpotent finite
  algebra $F$ such that $K=\varphi^{-1}(\varphi(K))$.
\end{Lemma}

\begin{proof}
  Let $\sigma$ be the typed shape of the elements of~$K$. At each node
  $\tau$ of $\sigma$, $K$ is determined by a clopen subset $K_\tau$ of
  either $X$ or else of $\Omega_{n_\tau}$ for a node with $n_\tau\ge0$
  children; to uniformize the notation, we let $n_\tau=-1$ and
  $\Omega_{-1}=X$ if the space corresponding to the node $\tau$
  is~$X$. Since each $K_\tau$ is clopen and $\sigma$ is a finite tree,
  there is for each $n\ge-1$ a continuous retraction
  $\varphi_n:\Omega_n\to\Sigma_n$ onto a finite subset $\Sigma_n$
  of~$\Omega_n$ such that, for each node $\tau$ of~$\sigma$,
  $\varphi_{n_\tau}^{-1}(\varphi_{n_\tau}(K_\tau))=K_\tau$. Such a
  function may be obtained by collapsing to one of its points each
  atom in the Boolean subalgebra of $\Cl P_{co}(\Omega_n)$ generated
  by the clopen sets $K_\tau$ such that $n_\tau=n$.

  Let $Y=\Sigma_{-1}$ and $\Sigma$ be the discrete signature whose
  $n$-ary symbols are those of~$\Sigma_n$ for each $n\ge0$. We define
  $F$ to consist of the $\Sigma$-terms over the set $Y$ of typed shape
  $\sigma$ together with all their subterms and an additional element
  $\bot$. When the term operations applied to the elements of $F$
  produce a term in $F$, the operation is defined to give that value;
  otherwise, the value of the operation is~$\bot$. Then, $F$ is a
  nilpotent $\Sigma$-algebra and, through the continuous functions
  $\varphi_n$, we may view $F$ as a nilpotent topological
  $\Omega$-algebra which is generated by the composite mapping
  $X\xrightarrow{\varphi_{-1}}{} Y\hookrightarrow F$. This generating
  mapping extends uniquely to a continuous homomorphism
  $\varphi:T_\Omega(X)\to F$ which is readily seen to have the
  required property.
\end{proof}

\begin{Prop}
  \label{p:free-Stone-embedding-of-terms}
  Let $X$ and each $\Omega_n$ be Hausdorff 0-dimensional spaces. Then
  the topological algebra $T_\Omega (X)$ is residually finite.
  Moreover, for every Stone pseudovariety \Cl S containing $\Cl
  S_{\mathrm{nil}}\cap\pv{Fin}_\Omega$, $T_\Omega (X)$ is
  homeo\-morphic-isomorphic to the subalgebra of $\Om X{\Cl S}$
  generated by $X$.
\end{Prop}

\begin{proof}
  Consider distinct terms $s$ and $t$ in $T_\Omega(X)$. Since
  $T_\Omega(X)$ is a Hausdorff 0-dimensional space, there is a clopen
  subset $K$ separating $s$ and $t$. By definition of the topology
  of~$T_\Omega(X)$, $K$ is a union of clopen sets of constant typed
  shape, each of which is a product of clopen sets at the
  corresponding node sets. Hence, we may assume that $K$ is one of the
  clopen sets in the statement of Lemma~\ref{l:nilpotent-algebra}. The
  lemma yields a continuous homomorphism $\varphi:T_\Omega(X)\to F$
  into a finite algebra $F$ such that $K=\varphi^{-1}(\varphi(K))$
  which, therefore, separates the terms $s$ and $t$. This establishes
  that $T_\Omega(X)$ is residually finite.

  Let $\iota:T_\Omega(X)\to\Om X{\Cl S}$ be the natural continuous
  homomorphism given by
  Proposition~\ref{p:term-algebra-free-as-top-algebra} and let
  $\image\iota$ be the image of~$\iota$. To prove the last statement
  in the proposition it suffices to show that, for every clopen subset
  $K$ of~$T_\Omega(X)$ with the property considered in
  Lemma~\ref{l:nilpotent-algebra}, the set $\iota(K)$ is the
  intersection of an open subset of~$\Om X{\Cl S}$ with $\image\iota$.
  Let $\varphi:T_\Omega(X)\to F$ be the continuous homomorphism given
  by Lemma~\ref{l:nilpotent-algebra} and let $\hat{\varphi}:\Om
  X{\Cl S}\to F$ be the unique continuous homomorphism such that
  $\hat{\varphi}\circ\iota=\varphi$. To conclude the proof, it remains
  to observe that
  $\hat{\varphi}^{-1}(\varphi(K))\cap\image\iota=\iota(K)$.
\end{proof}

The universal property of \Cl S-free Stone topological algebras can be
somewhat extended as follows to consider residually \Cl S Stone
topological algebras.

\begin{Prop}
  \label{p:free-Stone2}
  Let \Cl S be a Stone pseudovariety and let $X$ be a topological
  space. Then, for every Stone topological algebra $S$ that is
  residually \Cl S and every continuous mapping $\varphi:X\to S$,
  there is a unique continuous homomorphism $\hat{\varphi}:\Om X{\Cl
    S}\to S$ such that $\hat{\varphi}\circ\iota=\varphi$, where
  $\iota:X\to\Om X{\Cl S}$ is the natural generating mapping.
\end{Prop}

\begin{proof}
  Without loss of generality, we may assume that $\varphi$ is a
  generating mapping. Since $S$ is residually \Cl S, there is a
  continuous embedding $\varepsilon:S\to\prod_{i\in I}S_i$ into a
  product of members of~\Cl S. Consider the following commutative
  diagram:
  \begin{displaymath}
    \xymatrix{
      X \ar[r]^(.4)\iota \ar[d]_\varphi
      & \Om X{\Cl S}
      \ar@{-->}[rd]^{\hat{\varphi}_i}
      \ar@{-->}[d]^{\Phi}
      \ar@{-->}[ld]_{\hat{\varphi}}
      & \\
      S \ar[r]_(.35)\varepsilon
      & \prod_{i\in I}S_i \ar[r]_(.6){\pi_i}
      & S_i }
  \end{displaymath}
  where the mapping $\pi_i$ is the component projection,
  $\hat{\varphi}_i$ is the unique continuous homomorphism such that
  the outer trapezoid commutes, given by the universal property
  defining $\Om X{\Cl S}$; $\Phi$ is the mapping into the product
  induced by the $\hat{\varphi}_i$, that is, such that the right
  triangle commutes for every $i\in I$; and the existence of the
  mapping $\hat{\varphi}$ follows from the fact that both the images
  of~$\Phi$ and $\varepsilon$ in the product $\prod_{i\in I}S_i$ are
  $X$-generated and, therefore they are equal.
\end{proof}

In particular, we deduce that there is an onto continuous homomorphism
$\eta : \Om X{}\Cl S \rightarrow \Om X{}\Cl S'$ respecting generators
whenever the pair of Stone pseudovarieties satisfies $\Cl S' \subseteq
\Cl S$. We call it the \emph{natural continuous homomorphism} $\Om
X{}\Cl S \rightarrow \Om X{}\Cl S'$.

\subsection{Free Stone topological algebras are not profinite}
\label{sec:free-Stone-not-profinite}

We may use the space $\beta\mathbb{N}$ to show that free Stone
topological algebras are almost never profinite.

\begin{Thm}
  \label{t:Stone-free-not-profinite}
  If $X$ is a nonempty topological space and $\Omega\ne\Omega_0$ is an
  arbitrary topological signature, then the Stone topological algebra
  $\Om X{\St}_\Omega$ is not profinite.
\end{Thm}

\begin{proof}
  By assumption, we may choose $n\ge1$ such that
  $\Omega_n\ne\emptyset$. Let $u\in\Omega_n$ be fixed and consider the
  signature $\Omega'=\Omega'_n=\{u\}$.

  We define a structure of topological $\Omega'$-algebra on
  $\beta\mathbb{N}$ by the evaluation mapping
  \begin{displaymath}
    E_n:\Omega'_n\times(\beta\mathbb{N})^n\to\beta\mathbb{N}
  \end{displaymath}
  which is given by the projection on the last component followed by
  addition by~1. Here, by addition by~1 we mean the unique extension
  to a continuous mapping $\beta\mathbb{N}\to\beta\mathbb{N}$ of
  addition by~1 viewed as a mapping $\mathbb{N}\to\beta\mathbb{N}$. As
  $\beta\mathbb{N}$ is a Stone space
  (cf.~\cite[Proposition~3.9]{Walker:1974}), it is thus a Stone
  topological $\Omega'$-algebra, which is generated by the
  element~$0$. We claim that it is not profinite. Indeed, since
  $\Omega'$~is finite, up to isomorphism there are only countably many
  finite $\Omega'$-algebras. As $\beta\mathbb{N}$ is finitely
  generated, there are only finitely many continuous homomorphisms
  into a specific finite algebra. Hence, if $\beta\mathbb{N}$ would be
  residually finite, then it would embed into an algebra of
  cardinality $2^{\aleph_0}$, which contradicts the fact that its
  cardinality is really $2^{2^{\aleph_0}}$
  (cf.~\cite[Corollary~3.6.12]{Engelking:1989}). Hence, the
  $\Omega'$-algebra $\beta\mathbb{N}$ is not profinite.

  We may also view $\beta\mathbb{N}$ as a Stone topological
  $\Omega$-algebra by interpreting every operation different from $u$
  as the constant operation with value $0$. If this enriched structure
  would be profinite, so would be the original one.
  
  By the universal property of $\Om X{\St_\Omega}$, the constant
  mapping $\varphi:X\to\beta\mathbb{N}$ with value $0$ determines a
  unique continuous homomorphism $\hat{\varphi}:\Om
  X{\St_\Omega}\to\beta\mathbb{N}$ such that
  $\hat{\varphi}\circ\iota=\varphi$. Since $\beta\mathbb{N}$ is
  generated by the image of~$\varphi$, the mapping $\hat{\varphi}$ is
  onto. Since $\beta\mathbb{N}$ is not profinite as a topological
  $\Omega$-algebra, by Theorem~\ref{t:quotient-profinite} we conclude
  that neither is $\Om X{\St}_\Omega$, as desired.
\end{proof}

\subsection{Finite quotients}
\label{sec:quotients}

One of the basic observations in the theory of profinite algebras is
that all finite quotients of $\Om XV$ belong to $\pv V$. We next show
an analog of this statement in the realm of Stone pseudovarieties. We
start with a technical observation which extends a well known result
on pseudovarieties of finite algebras \cite[Lemma~4.1]{Almeida:2002a}.

\begin{Lemma}
  \label{l:recognition}
  Let $\mathcal{S}$ be a Stone pseudovariety and $S$ be a Stone
  topological algebra which is residually $\mathcal S$. If $K$ is a
  clopen subset of~$S$, then there exist $T\in \mathcal S$ and a
  continuous homomorphism $\varphi : S \rightarrow T$ such that the
  equality $\varphi^{-1}(\varphi(K))=K$ holds.
\end{Lemma}

\begin{proof}
  Let $s\in K$. Then, for each $u\in S\setminus K$ there is a
  continuous homomorphism $\alpha_u : S \rightarrow M_u$ with
  $M_u\in\mathcal S$ such that $\alpha_u(s)\not =\alpha_u(u)$. Since
  $\{\alpha_u(s)\}$ is closed in $M_u$, we obtain an open subset
  $O_u=\alpha^{-1}_u(M_u \setminus \{\alpha_u(s)\})$ of $S$. We see
  that $s\not\in O_u$ and $u\in O_u$. In this way we get an open cover
  $\{O_u\}_{u\in S\setminus K}$ of the closed subset $S\setminus K$ of
  the compact space $S$. It follows that there is a finite set
  $F\subseteq S\setminus K$ such that $S\setminus K \subseteq
  \bigcup_{u\in F} O_u$. We consider the finite product
  $T_s=\prod_{u\in F} M_u\in\mathcal S$ and the corresponding
  continuous homomorphism $\gamma_s : S \rightarrow T_s$. We claim
  that $\gamma_s(s)\not\in \gamma_s (S\setminus K)$. Indeed, for $x\in
  S\setminus K$ there is $u\in F$ such that $x\in O_u$ and so
  $(\gamma_s(x))_u=\alpha_u(x)\in M_u \setminus \{\alpha_u(s)\}$ and
  $\alpha_u(x)\not=\alpha_u(s)$ follows. This means that
  $\gamma_s(x)\not =\gamma_s(s)$.
  
  Since $S$ and $T_s$ are compact, $\gamma_s(S\setminus K)$ is closed.
  Hence there is an open subset $U_s\subseteq T_s$ such that
  $\gamma_s(s)\in U_s$ and $U_s \cap \gamma_s(S\setminus
  K)=\emptyset$. We denote $L_s=\gamma_s^{-1}(U_s)$ for which we have
  $\gamma_s (L_s) \cap \gamma_s (S\setminus K) =\emptyset$ and $s\in
  L_s$. Since there is such an open set $L_s$ for every $s\in K$, we
  may choose a finite subcover from the open cover $(L_s)_{s\in K}$ of
  $K$. Hence, there is a finite subset $G\subseteq K$ such that $K$ is
  covered by $(L_s)_{s\in G}$. Again, we consider the finite product
  $T=\prod_{s\in G} T_s\in\mathcal S$ and the corresponding continuous
  homomorphism $\varphi : S \rightarrow T$ whose components are the
  $\gamma_s$ with $s\in G$. It remains to show that
  $\varphi^{-1}(\varphi(K))=K$. The inclusion $K\subseteq
  \varphi^{-1}(\varphi(K))$ being trivial, assume for a moment that
  $\varphi^{-1}(\varphi(K))\not\subseteq K$. Then there are $y\not\in
  K$ and $z\in K$ such that $\varphi(y)=\varphi(z)$. If we take some
  $s\in G$ satisfying $z\in L_s$, then
  $(\varphi(y))_s=(\varphi(z))_s$ gives $\gamma_s(y)=\gamma_s(z)$.
  Since $y\in S\setminus K$, we get $\gamma_s(z)\in\gamma_s(S\setminus
  K)$ which contradicts the observation that $\gamma_s (L_s) \cap
  \gamma_s (S\setminus K) =\emptyset$. Hence, we conclude that
  $\varphi^{-1}(\varphi(K))\subseteq K$.
\end{proof}

Now we are ready to prove a statement concerning finite quotients of 
a Stone topological algebra which is residually \Cl S.

\begin{Prop}
  \label{p:finite-quotient-residually-S}
  Let $\mathcal{S}$ be a Stone pseudovariety and $S$ be a Stone
  topological algebra which is residually $\mathcal S$. Then every
  finite quotient of $S$ belongs to~\Cl S.
\end{Prop}

\begin{proof}
  Let $\alpha : S \rightarrow F$ be a continuous homomorphism onto a
  finite discrete algebra $F$. For each element $f\in F$ we may
  consider the clopen set $\alpha^{-1}(f)$ and a continuous
  homomorphism $\varphi_f : S \rightarrow T_f$ into a member of~\Cl S
  such that
  $\alpha^{-1}(f)=\varphi_f^{-1}(\varphi_f(\alpha^{-1}(f)))$, whose
  existence is ensured by Lemma~\ref{l:recognition}. Let $T$ be the
  finite product $\prod_{f\in F}T_f$, which lies in~\Cl S, and
  $\varphi : S \rightarrow T$ the corresponding continuous
  homomorphism, whose components are the $\varphi_f$ ($f\in F$). We
  let $T'$ be the image of $\varphi$, which also belongs to~\Cl S.

  We claim that if two elements of $S$ are identified by~$\varphi$,
  then they are also identified by~$\alpha$. For this purpose, let
  $s,t\in S$ be a pair of elements such that $\varphi(s)=\varphi(t)$.
  We take $g=\alpha(s)$ and recall that
  $\varphi_g^{-1}(\varphi_g(\alpha^{-1}(g)))=\alpha^{-1}(g)$. Since
  $\varphi_g(t)=\varphi_g(s)\in \varphi_g(\alpha^{-1}(g))$, we see
  that $t\in\alpha^{-1}(g)$. Hence, we have $\alpha(t)=g=\alpha(s)$,
  which establishes the claim. It follows that there is a homomorphism
  of algebras $\beta : T' \rightarrow F$ satisfying $\beta \circ
  \varphi = \alpha$, as depicted in the following commutative diagram,
  where $\pi_f$ denotes a component projection.
  \begin{displaymath}
    \xymatrix{
      F
      & S \ar[l]_\alpha \ar[r]^{\varphi_f} \ar[d]^\varphi
      & T_f \\
      & T' \ar@{-->}[lu]^\beta \ar@{^{(((}->}[r]
      & T \ar[u]_{\pi_f}
    }
  \end{displaymath}
  Moreover, for every $f\in F$ we know that $\beta^{-1}
  (f)=\varphi(\alpha^{-1}(f))$ is a closed subset of $T'$ because
  $\alpha^{-1}(f)$ is clopen in $S$. Hence the preimage of any
  (closed) subset of $F$ is closed so that $\beta$ is a continuous
  homomorphism. Since we already saw that $\beta$ is onto and
  $T'\in\Cl S$, the proof is complete.
\end{proof}

\section{Preliminaries on Boolean algebra}
\label{sec:prel-bool-algebra}

Before proceeding in the next section with our approach to duality, we
present here the required preliminaries on Boolean algebras. This
includes three versions of continuity conditions involving the
inverses of evaluation mappings on a topological algebra.

\subsection{Boolean algebras of sets}
\label{sec:Boole-sets}

Given Boolean subalgebras $\Cl B_i$ of $\Cl P(S_i)$ ($i=1,\ldots,n$),
we may consider the Boolean subalgebra $\bigoplus_{i=1}^n\Cl B_i$ of
$\Cl P(S_1\times\cdots\times S_n)$ generated by all \emph{boxes} of
the form
\begin{displaymath}
  L_1\times \cdots\times L_n\ (L_i\in\Cl B_i).
\end{displaymath}
Note that, if each $\Cl B_i$ is the Boolean algebra of all clopen
subsets of a Stone space $S_i$, then the members
of~$\bigoplus_{i=1}^n\Cl B_i$ are the clopen subsets of the product
space $\prod_{i=1}^nS_i$. The construction $\bigoplus_{i=1}^n\Cl B_i$
is a concrete realization of what is called in the Boolean algebra
literature the \emph{sum} or the \emph{tensor product} of the Boolean
algebras $\Cl B_i$. Since the intersection of two boxes is a box and
the complement in $S_1\times\cdots\times S_n$ of a box is a finite
union of pairwise disjoint of boxes, the members
of~$\bigoplus_{i=1}^n\Cl B_i$ are finite unions of pairwise disjoint
boxes. In fact, for each element $P$ of~$\bigoplus_{i=1}^n\Cl B_i$, we
may find finite partitions
\begin{equation}
  \label{eq:ppd-partition}
  \stepcounter{equation}
  \tag{\theequation${}_i$} S_i=\biguplus_{j\in J_i}L_{i,j}\
  (L_{i,j}\in\Cl B_i),
\end{equation}
called a \emph{$\Cl B_i$-partition}, ($i=1,\ldots,n$) such that $P$ is
the union of a set of boxes of the form $L_{1,j_1}\times \cdots\times
L_{n,j_n}$: we can take the $L_{i,j}$ to be the atoms in the Boolean
algebra of subsets $S_i$ generated by the $i$th component projections
of the given boxes defining $P$. Such a list
$\bigl(\eqref{eq:ppd-partition}: i=1,\ldots,n\bigr)$ of partitions is
said to be a \emph{$(\Cl B_i)_i$-mesh} for~$P$. In case all $S_i$ are
equal and all $\Cl B_i$ are equal to \Cl B, then we refer simply to a
\emph{\Cl B-mesh} for~$P$.

We proceed to present an alternative way of finding a \Cl B-mesh for
$P$ that in a sense gives a minimum choice, which is independent of
any particular decomposition of $P$ as a union of boxes.

Given a subset $P$ of the Cartesian product $S_1\times\cdots\times
S_n$, we consider the binary relation on $S_i$ defined by
$s_i\mathrel{\rho_i}s_i'$ if, for all $s_j\in S_j$ ($j\ne i$), the
following equivalence holds:
\begin{displaymath}
  (s_1,\ldots,s_{i-1},s_i,s_{i+1},\ldots,s_n)\in P
  \iff
  (s_1,\ldots,s_{i-1},s_i',s_{i+1},\ldots,s_n)\in P.
\end{displaymath}
Note that $\rho_i$ is an equivalence relation on~$S_i$. We denote by
$s_i/\rho_i$ the $\rho_i$-class of~$s_i$.

\begin{Lemma}
  \label{l:canonical-decomposition}
  Let $\Cl B_i$ be a Boolean algebra of subsets of~$S_i$
  ($i=1,\ldots,n$) and suppose that $P$ is an element of
  $\bigoplus_{i=1}^n\Cl B_i$. Consider a list of partitions
  \eqref{eq:ppd-partition} ($i=1,\ldots,n$) defining a $(\Cl
  B_i)_i$-mesh for~$P$ and the equivalence relations $\rho_i$ defined
  above. Then, the following hold:
  \begin{enumerate}[(1)]
  \item\label{item:canonical-decomposition-1} if $s,t\in L_{i,j}$ for
    some indices $i$ and $j$, then $s\mathrel{\rho_i}t$;
  \item\label{item:canonical-decomposition-2}
    each equivalence relation $\rho_i$ has finite index;
  \item\label{item:canonical-decomposition-3}
    the classes of each $\rho_i$ belong to~$\Cl B_i$;
  \item\label{item:canonical-decomposition-4}
    we have a finite disjoint decomposition
    \begin{equation}
      \label{eq:canonical-decomposition-0}
      P=\bigcup_{(s_1,\ldots,s_n)\in P}(s_1/\rho_1)\times\cdots\times(s_n/\rho_n).
    \end{equation}
  \end{enumerate}
\end{Lemma}

\begin{proof}
  (\ref{item:canonical-decomposition-1}) Suppose that $s_k\in S_k$
  ($k\ne i$) are such that
  \begin{equation}
    \label{eq:canonical-decomposition-1}
      (s_1,\ldots,s_{i-1},s,s_{i+1},\ldots,s_n)\in P.
  \end{equation}
  By the definition of mesh for~$P$, there exist $j_k\in J_k$
  ($k=1,\ldots,n$) such that
  \begin{displaymath}
    (s_1,\ldots,s_{i-1},s,s_{i+1},\ldots,s_n)
    \in
    L_{1,j_1}\times\cdots\times L_{n,j_n}\subseteq P.
  \end{displaymath}
  Because \eqref{eq:ppd-partition} is a partition of~$S_i$, it follows
  that $j_i=j$ and so
  \begin{equation}
    \label{eq:canonical-decomposition-2}
      (s_1,\ldots,s_{i-1},t,s_{i+1},\ldots,s_n)\in P.
  \end{equation}
  By symmetry, \eqref{eq:canonical-decomposition-2} also
  implies~\eqref{eq:canonical-decomposition-1}, thereby establishing
  that $s\mathrel{\rho_i}t$.
  
  Properties (\ref{item:canonical-decomposition-2}) and
  (\ref{item:canonical-decomposition-3}) are an immediate consequence
  of~(\ref{item:canonical-decomposition-1}) as the partition of~$S_i$
  determined by~$\rho_i$ is refined by the finite partition
  \eqref{eq:ppd-partition}.

  (\ref{item:canonical-decomposition-4}) The inclusion from left to
  right in~\eqref{eq:canonical-decomposition-0} follows from the
  obvious fact that $(s_1,\ldots,s_n)$ belongs
  to~$(s_1/\rho_1)\times\cdots\times(s_n/\rho_n)$. For the reverse
  inclusion, suppose that $(s_1,\ldots,s_n)\in P$. We claim that, if
  $s_i\mathrel{\rho_i}t_i$ ($i=1,\ldots,n$), then $(t_1,\ldots,t_n)\in
  P$. To prove the claim, we establish by induction on $i$ that
  \begin{equation}
    \label{eq:canonical-decomposition-3}
    (t_1,\ldots,t_i,s_{i+1},\ldots,s_n)\in P.
  \end{equation}
  Indeed, \eqref{eq:canonical-decomposition-3} holds for $i=0$ by
  hypothesis. Assuming that \eqref{eq:canonical-decomposition-3} holds
  for a given $i<n$ then, since $s_{i+1}\mathrel{\rho_{i+1}}t_{i+1}$,
  we may replace $s_{i+1}$ by $t_{i+1}$ to conclude that
  \eqref{eq:canonical-decomposition-3} also holds for $i+1$.
\end{proof}

The following application of Lemma~\ref{l:canonical-decomposition}
plays a key role in the sequel of this section.

\begin{Cor}
  \label{c:oplus-vs-cap}
  Suppose that $(\Cl B_{i,j})_{j\in J}$ is a nonempty family of
  Boolean algebras of subsets of the set~$S_i$ ($i=1,\ldots,n$). Then,
  the following equality holds:
  \begin{equation}
    \label{eq:oplus-vs-cap}
    \bigcap_{j\in J}\Bigl(\bigoplus_{i=1}^n\Cl B_{i,j}\Bigr)
    =\bigoplus_{i=1}^n\Bigl(\bigcap_{j\in J}\Cl B_{i,j}\Bigr).
  \end{equation}
\end{Cor}

\begin{proof}
  The inclusion from right to left in~\eqref{eq:oplus-vs-cap} is
  immediate. For the reverse inclusion, suppose that $P$ is an element
  of the left hand side of~\eqref{eq:oplus-vs-cap}.

  Consider the equivalence relations $\rho_i$ defined above for the
  set~$P$. Let \Cl M be the mesh for~$P$ whose partition of~$S_i$
  consists of the classes of~$\rho_i$. By
  Lemma~\ref{l:canonical-decomposition}, \Cl M is a $(\Cl
  B_{i,j})_i$-mesh for~$P$ for every $j$ and, therefore, it is a
  $(\bigcap_{j\in J}\Cl B_{i,j})_i$-mesh for~$P$, which entails that
  $P$ belongs to the right hand side of~\eqref{eq:oplus-vs-cap}.
\end{proof}

We conclude this subsection with another technical result of a
combinatorial nature that is used later.

\begin{Lemma}
  \label{l:box-in-union-of-boxes}
  Let $P=\bigcup_{r\in R}B_r$ be a union of boxes
  $B_r=L_{1,r}\times\cdots\times L_{n,r}$ such that, for every
  $i\in\{1,\ldots,n\}$ and $r,s\in R$, the sets $L_{i,r}$ and
  $L_{i,s}$ are either disjoint or equal. For each
  $i\in\{1,\ldots,n\}$, let $r_i\in R$ and let $J_i$ be a nonempty
  subset of~$L_{i,r_i}$. If $J_1\times\cdots\times J_n$ is contained
  in $P$, then it is contained in a box $B_r$ for some $r\in R$.
\end{Lemma}

\begin{proof}
  Note that, for $r,s\in R$, the boxes $B_r$ and $B_s$ are either
  equal or disjoint. We proceed by induction on~$n$. The case $n=1$
  being trivial, assume the result holds with $n-1$ in the place of
  $n$. For each $r\in R$, consider the box
  $C_r=L_{1,r}\times\cdots\times L_{n-1,r}$. Then, projecting on the
  first $n-1$ components, we conclude that there is at least one index
  $r\in R$ such that $J_1\times\cdots\times J_{n-1}$ is contained in
  the box $C_r$. There may be more than one index $r\in R$ with that
  property but the box $C=C_r$ containing $J_1\times\cdots\times
  J_{n-1}$ is unique and in fact if a box $C_r$ intersects
  $J_1\times\cdots\times J_{n-1}$ nontrivially, then it must be equal
  to~$C$. It follows that there must exist $r\in R$ such that
  $L_{n,r}=L_{n,r_n}$ and $C_r=C$, so that $J_1\times\cdots\times
  J_n\subseteq B_r$, which achieves the induction step.
\end{proof}

\subsection{An algebraized continuity condition}
\label{sec:continuity}

Let $S$ be an $\Omega$-algebra. Given a Boolean subalgebra \Cl B of
$\Cl P(S)$ and $n\ge0$, let %
$\Cl B_n'=\Cl P_{co}(\Omega_n)\oplus\Cl B^{(n)}$ %
where $\Cl B^{(n)}=\bigoplus_{i=1}^n\Cl B$. %
Consider the following property of~\Cl B, where $E_n=E_n^S$ is the
evaluation mapping of~$S$:
\begin{Cequation}
  \label{eq:star}
  L\in\Cl B \implies
  \forall n\ge0,\ E_n^{-1}(L)\in\Cl B_n'.
\end{Cequation}
This expresses a sort of continuity condition but may be also viewed
algebraically as stating that each mapping $E_n^{-1}$ determines a
homomorphism of Boolean algebras $\Cl B\to\Cl B_n'$.

Note that the trivial Boolean algebra $\{\emptyset,S\}$
satisfies~\eqref{eq:star}.

\begin{Lemma}
  \label{l:max}
  Let \Cl A be a Boolean subalgebra of~$\Cl P(S)$. Then, there is a
  maximum Boolean subalgebra of~\Cl A satisfying~(\ref{eq:star}).
\end{Lemma}

\begin{proof}
  Suppose $\mathfrak B$ is a chain of Boolean subalgebras %
  of~\Cl A each of which satisfies~\eqref{eq:star} and let %
  $\Cl B=\bigcup_{\Cl D\in\mathfrak B}\Cl D$. It is routine to check
  that \Cl B is a Boolean subalgebra of~\Cl A. If $L\in\Cl B$ then
  there exists $\Cl D\in\mathfrak B$ such that $L\in\Cl D$, whence
  $E_n^{-1}(L)\in\Cl D'_n$ because 
  \Cl D satisfies~\eqref{eq:star}. Since
  $\Cl D_n'\subseteq\Cl B_n'$, it follows that $E_n^{-1}(L)\in\Cl B_n'$.
  Hence, \Cl B also satisfies~\eqref{eq:star}. By Zorn's Lemma, we
  conclude that there are maximal Boolean subalgebras of~\Cl A
  satisfying~\eqref{eq:star}.
  
  Next, suppose that $\Cl C$ and $\Cl D$ are Boolean subalgebras %
  of~\Cl A satisfying~\eqref{eq:star} and let $\Cl B$ be the least
  Boolean subalgebra of~\Cl A containing $\Cl C\cup\Cl D$. Each $L\in
  \Cl B$ is a finite union of sets of the form $M\cap N$ such that
  $M\in \Cl C$ and $N\in\Cl D$. For such an intersection $M\cap N$, we
  have
  \begin{displaymath}
    E_n^{-1}(M\cap N)
    =E_n^{-1}(M)\cap E_n^{-1}(N)
    \in \Cl C'_n\vee\Cl D'_n
    \subseteq\Cl B_n',
  \end{displaymath}
  since $\Cl C$ and $\Cl D$ satisfy~\eqref{eq:star}, where $\Cl
  C'_n\vee\Cl D'_n$ denotes the least Boolean subalgebra of~$\Cl
  P_{co}(\Omega_n)\oplus\Cl P(S)^{(n)}$ containing $\Cl C'_n\cup\Cl
  D'_n$. For a finite union of sets satisfying~\eqref{eq:star}, we use
  a similar equality for union, and we conclude that $\Cl B$
  satisfies~\eqref{eq:star}. Hence, there is only one maximal Boolean
  subalgebra satisfying~\eqref{eq:star}, which establishes the lemma.
\end{proof}

Denote by $\Cl B_{\mathrm{max}1}^S$ the maximum Boolean subalgebra
of~$\Cl P_{co}(S)$ satisfying~\eqref{eq:star}.

The following is an immediate application of
Corollary~\ref{c:oplus-vs-cap}.

\begin{Cor}
  \label{c:Bcap}
  If the nonempty family $(\Cl B_i)_{i\in I}$ of Boolean subalgebras
  of $\Cl P(S)$ satisfies~(\ref{eq:star}) then so does $\bigcap_{i\in
    I}\Cl B_i$.
\end{Cor}

\begin{proof}
  Let $L$ be an arbitrary element of~$\bigcap_{i\in I}\Cl B_i$. By the
  assumption that each $\Cl B_i$ satisfies~\eqref{eq:star}, we know
  that $E_n^{-1}(L)$ belongs to $(\Cl B_i)'_n$. By
  Corollary~\ref{c:oplus-vs-cap}, we have %
  $\bigcap_{i\in I}(\Cl B_i)'_n=(\bigcap_{i\in I}\Cl B_i)'_n$. Thus,
  $E_n^{-1}(L)$ belongs to %
  $(\bigcap_{i\in I}\Cl B_i)'_n$. Hence, the Boolean algebra
  $\bigcap_{i\in I}\Cl B_i$ satisfies~\eqref{eq:star}.
\end{proof}

The following result gives a special case in which it is easy to
identify the Boolean algebra $\Cl B_{\mathrm{max}1}^S$.

\begin{Prop}
  \label{p:Bmax-in-Stone}
  Let $S$ be a Stone topological $\Omega$-algebra where $\Omega$ is a
  Stone signature. Then the equality $\Cl B_{\mathrm{max}1}^S=\Cl
  P_{co}(S)$ holds.
\end{Prop}

\begin{proof}
  By continuity of~$E_n$, given $K\in\Cl P_{co}(S)$, the subset
  $E_n^{-1}(K)$ of the Stone space $\Omega_n\times S^n$ is clopen and,
  therefore, it is a union of boxes of the form
  $K_0\times K_1\times\cdots\times K_n$ with clopen sides $K_i$. Since
  $E_n^{-1}(K)$ is compact, it is a finite union of such
  boxes, which shows that $\Cl P_{co}(S)$ satisfies
  condition~(\ref{eq:star}).
\end{proof}

The following example shows that the assumption that $\Omega$ is a
Stone signature cannot be dropped in
Proposition~\ref{p:Bmax-in-Stone}.

\begin{eg}
  \label{eg:Bmax-proper}
  Let $S=\mathbb{N}\cup\{\infty\}$ be the one-point
  compactification of the set of all natural numbers, which is viewed
  as a discrete space. Consider the unary signature
  $\Omega=\Omega_1=\{a_i:i\in\mathbb{N}\}$, also viewed as a discrete
  space. We define a structure of $\Omega$-algebra on~$S$ through the
  evaluation mapping $E_1:\Omega\times S\to S$ given by
  \begin{align*}
    E_1(a_i, n) &= \max\{0,n-i\}\quad (n\in\mathbb{N})\\
    E_1(a_i, \infty) &= \infty.
  \end{align*}
   For a finite subset $K$ of~$\mathbb{N}$  (and these together
  with their complements in~$S$ are the clopen subsets of $S$), note
  that
  \begin{displaymath}
    E_1^{-1}(K)=\bigcup_{i\in\mathbb{N}} \{a_i\}\times(K+i)
  \end{displaymath}
  where $K+i=\{m+i:m\in K\}$. Since $K+i\ne K+j$ for $i\ne j$,
  $E_1^{-1}(K)$ cannot be expressed as a finite union of
  boxes. Hence, the Boolean algebra $\Cl P_{co}(S)$ does not
  satisfy condition~\eqref{eq:star}.
\end{eg}

The following result shows how condition~(\ref{eq:star}) is affected
by continuous homomorphisms between topological algebras.

\begin{Prop}
  \label{p:lifting-4.7}
  Let $\varphi:S\to T$ be a homomorphism of $\Omega$-algebras and
  let \Cl B be a Boolean algebra of subsets of~$T$. If \Cl B
  satisfies~(\ref{eq:star}) then so does $\varphi^{-1}(\Cl B)$ and the
  converse also holds if $\varphi$ is onto.
\end{Prop}

\begin{proof}
  We have the following commutative diagram for each $n\ge0$:
  \begin{equation}
    \label{eq:lifting-4.7-1}
    \xymatrix@C=15mm{
      \Omega_n\times S^n
      \ar[r]^(0.6){E_n^S}
      \ar[d]^{\mathrm{id}\times\varphi^n}
      &
      S
      \ar[d]^\varphi
      \\
      \Omega_n\times T^n
      \ar[r]^(0.6){E_n^T}
      &
      T
    }
  \end{equation}

  Suppose first that \Cl B satisfies~\eqref{eq:star} and let $L\in\Cl
  B$. By assumption, there is an expression of the form
  \begin{displaymath}
    (E_n^T)^{-1}(L)=\bigcup_{i=1}^r K_i\times
    L_{i,1}\times\cdots\times L_{i,n}
  \end{displaymath}
  where the $K_i$ are clopen subsets of~$\Omega_n$ and the $L_{i,j}$
  belong to~\Cl B. By the commutativity of
  Diagram~(\ref{eq:lifting-4.7-1}), we get
  \begin{align*}
    (E_n^S)^{-1}(\varphi^{-1}(L))
    &=(\mathrm{id}\times\varphi^n)^{-1}\left((E_n^T)^{-1}(L)\right)\\
    &=\bigcup_{i=1}^r K_i\times
      \varphi^{-1}(L_{i,1})\times\cdots\times \varphi^{-1}(L_{i,n}),
  \end{align*}
  which shows that $(E_n^S)^{-1}(\varphi^{-1}(L))$ belongs to
  $(\varphi^{-1}(\Cl B))'_n$. Hence, $\varphi^{-1}(\Cl B)$
  satisfies~\eqref{eq:star}.

  Conversely, suppose that $\varphi$ is onto and $\varphi^{-1}(\Cl B)$
  satisfies~\eqref{eq:star}, and let $L\in\Cl B$. By assumption, there
  is an expression
  \begin{equation}
    \label{eq:lifting-4.7-2}
    (E_n^S)^{-1}\left(\varphi^{-1}(L)\right)
    =\bigcup_{i=1}^r K_i\times
    \varphi^{-1}(L_{i,1})\times\cdots\times \varphi^{-1}(L_{i,n}),
  \end{equation}
  where the $K_i$ are clopen subsets of~$\Omega_n$ and the $L_{i,j}$
  belong to~\Cl B. Applying $\mathrm{id}\times\varphi^n$ to both sides
  of~(\ref{eq:lifting-4.7-2}), we get
  \begin{align*}
    (E_n^T)^{-1}(L)
    &=(\mathrm{id}\times\varphi^n)(\mathrm{id}\times\varphi^n)^{-1}
      \left((E_n^T)^{-1}(L)\right)\\
    &=(\mathrm{id}\times\varphi^n)
    \left((E_n^S)^{-1}\left(\varphi^{-1}(L)\right)\right)\\
    &=\bigcup_{i=1}^r K_i\times
    L_{i,1}\times\cdots\times L_{i,n}
  \end{align*}
  where the first equality follows from the hypothesis that $\varphi$
  is onto and the second one follows from the commutativity of
  Diagram~(\ref{eq:lifting-4.7-1}). Hence, \Cl B
  satisfies~\eqref{eq:star}.
\end{proof}

\subsection{A relaxed continuity condition}
\label{sec:relax-cont-cond}

While it follows from the results in the previous subsection together
with those in the next couple of subsections that
condition~\eqref{eq:star} adequately characterizes the Boolean
algebras whose dual Stone spaces have a natural structure of Stone
topological algebra in the case of a Stone signature,
Example~\ref{eg:Bmax-proper} shows that \eqref{eq:star} is no longer
adequate for a general topological signature. We proceed to introduce
a relaxed version of~\eqref{eq:star} that provides the desired
characterization.

Let \Cl B be a Boolean subalgebra of~$\Cl P(S)$ and $\Omega$ a
topological signature, where $S$ is an $\Omega$-algebra. Define $\Cl
B_n''$ to be the set of all unions of the form
\begin{equation}
  \label{eq:Bn''-decomposition}
  \bigcup_{r\in R}K_r\times L_{1,r}\times \cdots\times L_{n,r}
\end{equation}
such that the following conditions are satisfied:
\begin{enumerate}[({C2}.1)]
\item\label{item:Bn''-1} the sets $K_r$ form a clopen partition
  of~$\Omega_n$;
\item\label{item:Bn''-2} each $L_{k,r}$ belongs to~\Cl B;
\item\label{item:Bn''-3} for each $r_0\in R$, the set $\{r\in R:
  K_r=K_{r_0}\}$ is finite.
\end{enumerate}
We emphasize that in (C2.\ref{item:Bn''-1}), the function $r\mapsto
K_r$ may not be injective and (C2.\ref{item:Bn''-3}) expresses
precisely the property that each preimage is finite. Note that $\Cl
B_n'\subseteq\Cl B_n''$ while, by (C2.\ref{item:Bn''-1}), equality
holds in case $\Omega$ is a Stone signature. We prove below that in
general $\Cl B_n''$ is also a Boolean subalgebra of $\Cl
P(\Omega_n\times S^n)$.

There are two alternative and conflicting ways in which the
decomposition~\eqref{eq:Bn''-decomposition} may be rewritten which we
present in the next two lemmas.

\begin{Lemma}
  \label{l:Bn''-4}
  Every element of $\Cl B_n''$ admits a decomposition of the
  form~\eqref{eq:Bn''-decomposition} satisfying conditions
  (C2.\ref{item:Bn''-1})--(C2.\ref{item:Bn''-4}), where
  \begin{enumerate}[({C2}.1)]
    \setcounter{enumi}3
  \item\label{item:Bn''-4} for each $r_0\in R$ and each
    $i\in\{1,\ldots,n\}$, the sets of the form $L_{i,r}$ with $r$ such
    that $K_r=K_{r_0}$ are either disjoint or equal.
  \end{enumerate}
\end{Lemma}

\begin{proof}
  It suffices to note that, in view of~(C2.\ref{item:Bn''-3}), there
  are only finitely many boxes of the form considered
  in~(C2.\ref{item:Bn''-4}) and invoke
  Lemma~\ref{l:canonical-decomposition}.
\end{proof}

\begin{Lemma}
  \label{l:Bn''-5}
  The elements of $\Cl B_n''$ are the subsets of~$\Omega_n\times S^n$
  that admit decompositions of the form
  \begin{displaymath}
    \bigcup_{r\in R}K_r\times P_r
  \end{displaymath}
  where $r\to K_r$ is an injective function satisfying
  (C2.\ref{item:Bn''-1}) and each $P_r$ belongs to $\Cl B^{(n)}$.
\end{Lemma}

\begin{proof}
  Consider a decomposition \eqref{eq:Bn''-decomposition} of an element
  $U$ of~$\Cl B_n''$ satisfying
  (C2.\ref{item:Bn''-1})--(C2.\ref{item:Bn''-3}). Let $\sim$ be the
  equivalence relation on $R$ given by $r_1\sim r_2$ if
  $K_{r_1}=K_{r_2}$ and denote by $[r]$ the $\sim$-class
  containing~$r$. Then the decomposition
  \begin{displaymath}
    U=\bigcup_{[r]\in R/{\sim}}
    K_r\times(\bigcup_{s\in[r]}L_{1,s}\times \cdots\times L_{n,s})
  \end{displaymath}
  satisfies the requireed conditions as the inner union is finite by
  (C2.\ref{item:Bn''-3}) and each of its terms belongs to~$\Cl
  B^{(n)}$ by (C2.\ref{item:Bn''-2}). For the converse, it suffices to
  express each $P_r\in\Cl B^{(n)}$ as a finite union of boxes and
  distribute the Cartesian product of~$K_r$ over that union.
\end{proof}

When $A$ is a subset of a set $X$ and it is clear from the context
that it is considered as such, then we may write $A^\complement$
instead of $X\setminus A$.

\begin{Lemma}
  \label{l:Bn''-Boolean}
  If \Cl B is a Boolean subalgebra of~$\Cl P(S)$, then the set $\Cl
  B_n''$ is a Boolean subalgebra of $\Cl P(\Omega_n\times S^n)$.
\end{Lemma}

\begin{proof}
  We need to show that $\Cl B_n''$ is closed under finite union and
  complement. Consider two sets $A_1$ and $A_2$ in~$\Cl B_n''$, say
  with decompositions
  \begin{displaymath}
    A_i=\bigcup_{r\in R_i}K_{i,r}\times P_{i,r}
  \end{displaymath}
  satisfying the conditions from Lemma~\ref{l:Bn''-5}.  
  For the union
  $A_1\cup A_2$ we may find a similar decomposition by first taking
  the partition $\{K_r:r\in R\}$ of~$\Omega_n$ whose classes $K_r$ are
  the nonempty intersections of the form $K_{1,r_1}\cap K_{2,r_2}$
  with $r_1\in R_1$ and $r_2\in R_2$ and then taking the decomposition
  \begin{displaymath}
    \bigcup_{r\in R}
    (K_r\times P_{1,r_1}
    \cup
    K_r\times P_{2,r_2}).
  \end{displaymath}
  This shows that $A_1\cup A_2$ belongs to $\Cl B_n''$.

  For the complement $A_1^\complement$, we get the formula
  \begin{displaymath}
    A_1^\complement=\bigcup_{r\in R_1}K_{1,r}\times P_{1,r}^\complement 
  \end{displaymath}
  which ensures that $A_1^\complement \in \Cl B_n''$ by Lemma~\ref{l:Bn''-5}.
\end{proof}

Let $S$ be an $\Omega$-algebra. Consider the following condition on a
Boolean subalgebra \Cl B of $\Cl P(S)$:
\begin{Cequation}
  \label{eq:star2}
  L\in\Cl B
  \implies
  \forall{ n\ge0},\ (E_n^S)^{-1}(L)\in\Cl B_n''.
\end{Cequation}
Again, condition~\eqref{eq:star2} means that each mapping
$(E_n^S)^{-1}$ defines a Boolean algebra homomorphism $\Cl B\to\Cl
B_n''$.

Without any significant change, several results regarding
condition~\eqref{eq:star} can also be proved for
condition~\eqref{eq:star2}. We leave it to the reader to verify that
the omitted proofs can be straghtforwardly adapted.

\begin{Lemma}
  \label{l:max2}
  Let \Cl A be a Boolean subalgebra of~$\Cl P(S)$. Then, there is a
  maximum Boolean subalgebra of~\Cl A satisfying~(\ref{eq:star2}).
\end{Lemma}

In particular, there is also a maximum Boolean subalgebra of~$\Cl
P_{co}(S)$ satisfying~\eqref{eq:star2}, which we denote $\Cl
B_{\mathrm{max}2}^S$.

The next result is the analog of Proposition~\ref{p:Bmax-in-Stone} but
requires a much longer proof for an arbitrary topological signature.

\begin{Prop}
  \label{p:Bmax2-in-Stone}
  Let $S$ be a Stone topological $\Omega$-algebra where $\Omega$ is an
  arbitrary topological signature. Then the equality $\Cl
  B_{\mathrm{max}2}^S=\Cl P_{co}(S)$ holds.
\end{Prop}

\begin{proof}
  As, by definition, $\Cl B_{\mathrm{max}2}^S$ is contained in~$\Cl
  P_{co}(S)$, it remains to show that the Boolean algebra $\Cl
  P_{co}(S)$ satisfies~\eqref{eq:star2}. Let $K$ be a clopen subset
  of~$S$ and consider the binary relation $\rho$ on~$\Omega_n$ defined
  by
  \begin{displaymath}
    (o,o')\in \rho
    \text{ if }
    \forall s_1,\ldots,s_n\in S\
    \left(
      o_S(s_1,\ldots,s_n)\in K
      \iff
      o'_S(s_1,\ldots,s_n)\in K
    \right).
  \end{displaymath}
  Note that $\rho$ is an equivalence relation on~$\Omega_n$. We claim
  that its classes are open, whence they are clopen. To establish the
  claim, suppose that $(p_i)_{i\in I}$ is a net in the complement
  in~$\Omega_n$ of the class of~$o$ converging to~$p$; we need to show
  that $(o,p)\notin\rho$. Then, for each $i\in I$, there exist
  $s_{1,i},\ldots,s_{n,i}\in S$ such that exactly one of
  $o_S(s_{1,i},\ldots,s_{n,i})$ and $(p_i)_S(s_{1,i},\ldots,s_{n,i})$
  belongs to~$K$. Since we may replace $K$ by $K^\complement$ without
  changing $\rho$, by taking a suitable subnet we may assume that it
  is $(p_i)_S(s_{1,i},\ldots,s_{n,i})$ that belongs to~$K$. Because $S$ is
  compact, we may further assume that the net
  $(s_{1,i},\ldots,s_{n,i})_{i\in I}$ converges to $(s_1,\ldots,s_n)$
  in~$S^n$. As the evaluation mapping $E_n$ is assumed to be
  continuous and $K$ is clopen, it follows that
  $p_S(s_1,\ldots,s_n)\in K$ and $o_S(s_1,\ldots,s_n)\notin K$. Hence,
  the $\rho$-class of~$o$ is open.

  Because $E_n^S$ is continuous, the set $(E_n^S)^{-1}(K)$ is open
  and, therefore it is a union
  \begin{equation}
    \label{eq:Bmax2-in-Stone-1}
    \bigcup_{r\in R}K_r\times L_{1,r}\times\cdots\times L_{n,r}
  \end{equation}
  of boxes with the $K_r$ open subsets of~$\Omega_n$ and the $L_{k,r}$
  clopen in~$S$. If $o\in K_r$ and $(o,p)\in\rho$, then $\{p\}\times
  L_{1,r}\times\cdots\times L_{n,r}\subseteq (E_n^S)^{-1}(K)$ by the
  definition of~$\rho$. Hence, we may assume that each $K_r$ is a
  union of $\rho$-classes. Moreover, we may break up each $K_r$ into
  the $\rho$-classes contained in it, at the cost of splitting terms
  in the decomposition~\eqref{eq:Bmax2-in-Stone-1}. The resulting
  decomposition clearly satisfies
  conditions~(C2.\ref{item:Bn''-1})--(C2.\ref{item:Bn''-2}), it only
  remains to modify it by showing that some terms may be dropped so as
  to satisfy (C2.\ref{item:Bn''-3}). For that purpose, we may first
  note that a similar decomposition
  \begin{equation*}
    \label{eq:Bmax2-in-Stone-2}
    \bigcup_{s\in R'}K'_s\times L'_{1,s}\times\cdots\times L'_{n,s}
  \end{equation*}
  exists for $(E_n^S)^{-1}(K^\complement)$, where again each $K_s'$ is
  a $\rho$-class and the sets $L'_{k,s}$ are clopen in~$S$. For a
  fixed $o\in\Omega_n$, the boxes $L_{1,r}\times\cdots\times L_{n,r}$
  such that $o\in K_r$ together with the boxes
  $L'_{1,s}\times\cdots\times L'_{n,s}$ with $o\in K'_s$ constitute a
  clopen cover of the compact space $S^n$ in which the former are
  disjoint from the latter. Hence, the union
  \begin{displaymath}
    \bigcup_{o\in K_r}K_r\times L_{1,r}\times\cdots\times L_{n,r}
  \end{displaymath}
  coincides with a finite union involving only finitely many of the
  same terms. Applying this kind of reduction
  of~\eqref{eq:Bmax2-in-Stone-1} for a complete set of representatives
  of the $\rho$-classes, we get the desired decomposition
  of~$(E_n^S)^{-1}(K)$ satisfying conditions
  (C2.\ref{item:Bn''-1})--(C2.\ref{item:Bn''-3}), thereby completing the
  proof of the proposition.
\end{proof}

The following result is again an application of
Corollary~\ref{c:oplus-vs-cap}.

\begin{Cor}
  \label{c:Bcap2}
  If the nonempty family $(\Cl B_i)_{i\in I}$ of Boolean subalgebras
  of $\Cl P(S)$ satisfies~(\ref{eq:star2}) then so does $\bigcap_{i\in
    I}\Cl B_i$.
\end{Cor}

\begin{proof}
  Let $L$ be an arbitrary element of~$\Cl B=\bigcap_{i\in I}\Cl B_i$. By the
  assumption that each $\Cl B_i$ satisfies~\eqref{eq:star2}, we know
  that $E_n^{-1}(L)$ belongs to $(\Cl B_i)_n''$. For each $i\in I$,
  there is a decomposition
  \begin{equation}
    \label{eq:Bcap2-star2}
    (E_n^S)^{-1}(L)
    =\bigcup_{r\in R_i}K_{i,r}\times P_{i,r}
  \end{equation}
  given by Lemma~\ref{l:Bn''-5}. Given $o\in\Omega_n$, we conclude
  that there is a unique $r_i\in R_i$ such that
  $(E_n^S)^{-1}(L)\cap\{o\}\times S^n=\{o\}\times P_{i,r_i}$. By
  Corollary~\ref{c:oplus-vs-cap}, it follows that there is $P_o\in\Cl
  B^{(n)}$ such that $(E_n^S)^{-1}(L)\cap\{o\}\times S^n=\{o\}\times
  P_o$, that is, all the sets $P_{i,r_i}$ are equal and belong to~$\Cl
  B^{(n)}$. Thus, the decomposition \eqref{eq:Bcap2-star2} for every
  $i\in I$ already shows that $(E_n^S)^{-1}(L)$ belongs to $\Cl
  B_n''$. Hence, the Boolean algebra $\Cl B$
  satisfies~\eqref{eq:star2}.
\end{proof}

We conclude this subsection with the analog of
Proposition~\ref{p:lifting-4.7}, whose proof can be easily adpated to
handle condition~\eqref{eq:star2}.

\begin{Prop}
  \label{p:lifting-star2}
  Let $\varphi:S\to T$ be a homomorphism of\/ $\Omega$-algebras and
  let \Cl B be a Boolean algebra of subsets of~$T$. If \Cl B
  satisfies~(\ref{eq:star2}) then so does $\varphi^{-1}(\Cl B)$ and
  the converse also holds if $\varphi$ is onto.
\end{Prop}

\subsection{A simplified version of the continuity condition}
\label{sec:simple-cont-cond}

We show that condition~\eqref{eq:star2} may be simplified by 
omitting the partition assumption from condition (C2.\ref{item:Bn''-1}). 

Let \Cl B be a Boolean algebra and $\Omega$ a topological signature.
Define $\Cl B_n'''$ to be the set of all unions of the form
\begin{equation}
  \label{eq:Cn-decomposition}
  \bigcup_{r\in R}K_r\times L_{1,r}\times \cdots\times L_{n,r}
\end{equation}
such that the following conditions are satisfied:
\begin{enumerate}[({C3}.1)]
\item\label{item:Cn-1} each set $K_r$ is clopen;
\item\label{item:Cn-2} each set $L_{k,r}$ belongs to~\Cl B;
\item\label{item:Cn-3} for each $o\in\Omega_n$, the set 
$\{r \in R : o\in K_r\}$ is finite.
\end{enumerate}

Although this definition is simpler than the definition of $\Cl
B_n''$, it has some disadvantages. In particular, it is not clear how
to prove properties like Corollary~\ref{c:Bcap2}. The basic technical
problem is that $\Cl B_n'''$ need not be a Boolean algebra in general.
To see this, one may consider the Boolean algebra $\Cl B=\Cl P
(\mathbb N)$, $n=1$, and $\Omega_1=\mathbb N^*$ the one-point
compactification of $\mathbb N$, where the limit of any unbounded
sequence is the added point $\infty$. Then the relation ``less than''
$R_{<}$
belongs to $\Cl B_n'''$ as $R_{<}=\bigcup_{n\in \mathbb N}
\{n\}\times\{x \in \mathbb N : n < x\}$. However, one can check that
$R_<^\complement$ does not belong to $\Cl B_n'''$. Indeed, assuming
$R_<^\complement=\bigcup_{i\in I} K_i\times L_i$ satisfies
(C3.\ref{item:Cn-1}) and~(C3.\ref{item:Cn-3}), one can see that there
is some $i$ for which $L_i$ is infinite and $\infty \in K_i$, so that
$K_i$ contains some element $n\in\mathbb{N}$ which is less than some
member of~$L_i$, which contradicts the assumption that $K_i\times L_i$
is contained in $R_<^\complement$.
 
Although the set $\Cl B_n'''$ does not coincide with the Boolean
algebra $\Cl B_n''$, they are strongly related. In particular, the set
$\Cl B_n'''$ clearly contains $\Cl B_n''$ and we show that $\Cl B_n''$
is the maximum Boolean subalgebra of $\Cl B_n'''$ in the following
lemma.

\begin{Lemma}
  \label{l:Cn-generated-by-Bn''}
  Let $S$ be an arbitrary set, \Cl B a Boolean subalgebra of~$\Cl
  P(S)$, $n$ an integer, and $\Omega$ a topological signature. If a
  set $K$ in $\Cl B_n'''$ is such that $K^\complement$ also belongs to
  $\Cl B_n'''$, then $K$ belongs to $\Cl B_n''$.
\end{Lemma}

\begin{proof}
  Let $K\in \Cl B_n'''$ be given by
  formula~(\ref{eq:Cn-decomposition}) and suppose that
  $K^\complement\in\Cl B'''_n$. We apply the idea from the proof of
  Proposition~\ref{p:Bmax2-in-Stone}, namely, for this $K\subseteq
  \Omega_n\times S^n$ we consider the binary relation $\rho$
  on~$\Omega_n$ defined by
  \begin{displaymath}
    (o,o')\in \rho
    \text{ if }
    \forall s_1,\ldots,s_n\in S\
    \left(
      (o,s_1,\ldots,s_n)\in K
      \iff
      (o',s_1,\ldots,s_n)\in K
    \right).
  \end{displaymath}
  We see that $\rho$ is an equivalence relation and that we get the same relation, 
  if we replace the set $K$ by $K^\complement$ in the previous 
  definition.
  We claim that $\rho$ is a clopen equivalence on the set $\Omega_n$. 
  The proof of this claim occupies the following few paragraphs.
    
  For an arbitrary $o\in \Omega_n$, we consider $[o]_\rho=\{ p\in
  \Omega_n : (o,p)\in \rho\}$, the equivalence class containing $o$.
  To show that $[o]_\rho$ is open, let $(p_i)_{i\in I}$ be a net in
  $[o]_\rho^\complement$ converging to $p\in\Omega_n$. We improve the
  expression~(\ref{eq:Cn-decomposition}) for the considered set $K$;
  more precisely, we rewrite the part using indices from the subset
  $R'=\{r \in R : o\in K_r\ \vee\ p\in K_r\}$. Since the set $R'$ is
  finite, the subset $\{ L_{i,r} : i\in\{1,\dots ,n\},\ r\in R'\}$ of
  the Boolean algebra $\Cl B$ is finite as well. Thus there exists a
  finite partition $S=\biguplus_{j\in J}M_{j}$, indexed by a finite
  set of integers $J=\{1, \dots ,k\}$ such that each $M_j$ belongs
  to~\Cl B and, for every $r\in R'$ and $i=1,\dots ,n$, there exist an
  integer $m(i,r)$ and a sequence of indices $j_1, \dots ,j_{m(i,r)}$
  from $J$ such that $L_{i,r}=M_{j_1} \cup \dots \cup M_{j_{m(i,r)}}$.
  We replace each such $L_{i,r}$ in the considered
  expresion~(\ref{eq:Cn-decomposition}) for $K$ by this finite union
  $M_{j_1} \cup \dots \cup M_{j_{m(i,r)}}$ and distribute the product
  $K_r\times L_{1,r}\times \cdots\times L_{n,r}$ over each such union.
  In this way, the finite part $\bigcup_{r\in R'}K_r\times
  L_{1,r}\times \cdots\times L_{n,r}$ of the
  expression~(\ref{eq:Cn-decomposition}) is replaced by $\bigcup_{r\in
    R''} K_r\times P_r$, where $R''$ is a new finite index set, $K_r$
  are the clopen sets used in the original expression but indexed by
  the new index set $R''$, and every $P_r$ is a basic box of the form
  $P_{1,r}\times \dots \times P_{n,r}$ where all $P_{i,r}\in\Cl B$
  belong to the considered partition $\{M_{j}\}_{j\in J}$. In other
  words, if we define, for each $n$-tuple
  $\alpha=(\alpha_1,\ldots,\alpha_n)\in J^n$, the basic box $P_\alpha$
  by the formula $P_\alpha=M_{\alpha_1}\times \cdots \times
  M_{\alpha_n}$, then $P_r=P_\alpha$ for some $\alpha \in J^n$.
  
  The last step in our improvement is that for a fixed basic box
  $P_\alpha$, we put together all $K_r$'s such that $P_r=P_\alpha$.
  Formally, we put $K_\alpha = \bigcup_{r\in R'', P_r=P_\alpha} K_r$
  whenever the index set $\{r\in R'': P_r=P_\alpha\}$ is non-empty,
  and we put $K_\alpha=\emptyset$ otherwise. Then we replace the
  existing finite subparts of the considered expression of the form
  \begin{displaymath}
    \bigcup_{r\in R'', P_r=P_\alpha} K_r\times P_r 
    \quad \text{ by }\quad K_\alpha \times P_\alpha.
  \end{displaymath}
  Additionally, we add the empty sets $K_\alpha \times P_\alpha$ for
  those $\alpha$ for which the index set $\{r\in R'': P_r=P_\alpha\}$
  is empty. Altogether, we obtain the expression
  \begin{displaymath}
    K =
    \bigcup_{r\in R\setminus R'}
    K_r\times  L_{1,r}\times \cdots\times L_{n,r} 
    \cup
    \bigcup_{\alpha \in J^n}
    K_{\alpha} \times P_\alpha
  \end{displaymath}
  with the
  following properties:
  \begin{enumerate}[(i)]
  \item\label{item:Cn-generated-by-Bn''-1} $o,p\not\in K_r$ for all
    $r\in R\setminus R'$,
  \item\label{item:Cn-generated-by-Bn''-2} $P_\alpha \cap
    P_{\alpha'}\not=\emptyset$ if and only if $\alpha=\alpha'$, and
  \item\label{item:Cn-generated-by-Bn''-3} every $K_r$ (for $r\in
    R\setminus R'$) and $K_{\alpha}$ (for $\alpha\in J^n$) is clopen
    in $\Omega_n$.
  \end{enumerate}
  Notice that the same construction can be made for $K^\complement$
  independently.
  
  Now, recall that we assume $(p_i,o)\not\in \rho$ for every $i\in I$.
  For each $i\in I$, there is an $n$-tuple
  $(s_{1,i},\ldots,s_{n,i})\in S^n$ such that
  \begin{displaymath}
    \left|\,\{(o,s_{1,i},\ldots ,s_{n,i}), (p_i,s_{1,i},\ldots
      ,s_{n,i})\}\cap K\,\right| =1 .
  \end{displaymath}
  Whence, it is possible to choose a converging subnet of
  $(p_{i_\lambda})_{\lambda\in\Lambda}$ of the net $(p_i)_{i\in
    I}$
  such that, choosing $K'$ equal to $K$ or $K^\complement$, we have:
  \begin{displaymath}
    \forall \lambda\in\Lambda:  \quad
    \left(\,
      (o,s_{1,i_\lambda},\ldots ,s_{n,i_\lambda})\in K'
      \ \wedge\ 
      (p_{i_\lambda},s_{1,i_\lambda},\ldots ,s_{n,i_\lambda})\not\in
      K'\,
    \right).
  \end{displaymath}
  Without loss of generality, we may assume that $K'=K$. By the above
  property~(\ref{item:Cn-generated-by-Bn''-1}), the fact
  $(o,s_{1,i_\lambda},\dots ,s_{n,i_\lambda})\in K$ means that
  $(o,s_{1,i_\lambda},\dots ,s_{n,i_\lambda})\in K_\alpha\times
  P_\alpha$ for some $\alpha\in J^n$. Since $J^n$ is finite, by taking
  an appropriate subnet of the net $(p_{i_\lambda})_{\lambda\in\Lambda}$
  we may further assume that there is $\alpha\in J^n$ such that
  $(o,s_{1,i_\lambda},\dots ,s_{n,i_\lambda})\in K_\alpha\times P_\alpha$
  for all $\lambda\in\Lambda$. We deduce that $o\in K_\alpha$ and
  $p_{i_\lambda}\not\in K_\alpha$. Since the subset $K_\alpha$ is
  clopen by~(\ref{item:Cn-generated-by-Bn''-3}), and since $p$ is a
  limit of the net $(p_i)_{i\in I}$ and consequently also the limit of
  the considered subnet $(p_{i_\lambda})_{\lambda\in\Lambda}$, we get
  that $p\not\in K_\alpha$. We choose some element
  $(s_1,\ldots,s_n)\in P_\alpha$, which gives $(o,s_1,\ldots,s_n)\in
  K$. Recall that $(s_1,\ldots,s_n)\not\in P_{\alpha'}$ for
  $\alpha'\not=\alpha$ by~(\ref{item:Cn-generated-by-Bn''-2}).
  Finally, as $p$ does not belong to $K_r$ for $r\in R\setminus R'$,
  we deduce that $(p, s_1,\dots , s_n)\not\in K$. This leads to the
  conclusion $(o,p)\not\in\rho$. This completes the proof of the
  claim that each $\rho$-class is open; therefore, it is also clopen.
  
  Now we choose a set of representatives of all $\rho$-clasess, that
  is, a set $O\subseteq \Omega_n$ such that for every $o\in\Omega_n$
  there is a unique $o'\in O$ for which $(o',o)\in \rho$. For each
  $o\in O$, we consider the set $P_o \subseteq S^n$ determined by the
  property $K \cap (\{ o \}\times S^n) =\{o\}\times P_o$. By the
  condition~(C3.\ref{item:Cn-3}), we know that $P_o$ is a finite union
  of boxes with sides in $\mathcal B$, including the case when $P_o$
  is empty. Finally, we obtain the equality $K=\bigcup_{o\in O}
  [o]_\rho \times P_o$, which shows that $K\in \Cl B''_n$.
\end{proof}

We are prepared to give an alternative condition for~\eqref{eq:star2}.
Let $S$ be an $\Omega$-algebra. Consider the following condition on a
Boolean subalgebra \Cl B of $\Cl P(S)$:
\begin{Cequation}
  \label{eq:star3}
  L\in\Cl B \implies
  \forall{n\ge0},\ (E_n^S)^{-1}(L)\in\Cl B_n'''.
\end{Cequation}
The new condition is equivalent to condition~(\ref{eq:star2}), as we
show in the following statement.

\begin{Prop}
  \label{p:equvalence-c2-c3}
  Let $S$ be an $\Omega$-algebra and $\Cl B$ be a Boolean subalgebra
  of $\Cl P(S)$. Then $\Cl B$ satisfies~(\ref{eq:star2}) if and only
  if it satisfies~(\ref{eq:star3}).
\end{Prop}

\begin{proof} 
  If we consider the sets $\Cl B''_n$ and $\Cl B_n'''$ for a given
  Boolean subalgebra $\Cl B$ and an integer $n$, then we have $\Cl
  B''_n \subseteq \Cl B_n'''$ which yields that
  property~(\ref{eq:star2}) implies property~(\ref{eq:star3}).
  
  For the reverse implication, let $L\in\Cl B$ be arbitrary. If we
  denote $K=(E_n^S)^{-1}(L)\in\Cl B_n'''$, then the set
  $K^\complement=(E_n^S)^{-1}(L^\complement)$ also belongs to $\Cl
  B_n'''$ because $L^\complement\in \Cl B$. Hence we get $K\in\Cl
  B''_n$ by Lemma~\ref{l:Cn-generated-by-Bn''}.
\end{proof}

\section{Duality}
\label{sec:duality}

For a Boolean algebra \Cl B, denote $\Cl B^\star$ the Stone dual space
of~\Cl B. Recall that $\Cl B^\star$ may be viewed as the set of all
ultrafilters of~\Cl B; a basis of the topology is given by the sets
$\Cl U_L$ ($L\in\Cl B$), where $\Cl U_L$ consists of all ultrafilters
containing $L$.

For a set $S$, let $\Cl P(S)$ be the Boolean algebra of all subsets
of~$S$. If \Cl B is a Boolean subalgebra of $\Cl P(S)$, then we let
$\iota:S\to\Cl B^\star$ be defined by
\begin{displaymath}
  \iota(s)
  =s{\uparrow}
  =\{L\in\Cl B: s\in L\}.
\end{displaymath}
In case $S$ is a topological space, we let $\Cl P_{co}(S)$
denote the Boolean algebra of all clopen subsets of~$S$.

\begin{Prop}
  \label{p:basic-facts-on-Stone-dual}
  Let \Cl B be a boolean algebra of subsets of a set~$S$. Then the
  following properties hold for an arbitrary element $L$ of~\Cl B:
  \begin{enumerate}
  \item\label{item:basic-facts-on-Stone-dual-1}
    $\overline{\iota(L)}=\Cl U_L$;
  \item\label{item:basic-facts-on-Stone-dual-2}
    $\iota^{-1}\Bigl(\overline{\iota(L)}\Bigr)=L$.
  \end{enumerate}
  Moreover, the mapping $\varphi:L\mapsto\overline{\iota(L)}$ is an
  isomorphism of \Cl B with~$\Cl P_{co}(\Cl B^\star)$.
\end{Prop}

\begin{proof}
  (\ref{item:basic-facts-on-Stone-dual-1}) The closure
  $\overline{\iota(L)}$ consists of all ultrafilters $u\in\Cl B^\star$
  such that, for all $K\in\Cl B$, $u\in \Cl U_K$ implies $\Cl
  U_K\cap\iota(L)\ne\emptyset$. The former condition means that $K\in
  u$ while the latter means that there is $s\in L$ such that
  $s{\uparrow}\in\Cl U_K$, that is, $s\in K$, and thus it holds if and
  only if $K\cap L\ne\emptyset$. Since $u$~is an ultrafilter, $L$
  having nonempty intersection with all $K\in u$ is equivalent to
  $L\in u$, that is $u\in\Cl U_L$.

  (\ref{item:basic-facts-on-Stone-dual-2}) In view
  of~(\ref{item:basic-facts-on-Stone-dual-1}), we need to show that
  $\iota^{-1}(\Cl U_L)=L$. Indeed, $s\in\iota^{-1}(\Cl U_L)$ holds if
  and only if $s{\uparrow}\in\Cl U_L$, that is, $s\in L$.

  By~(\ref{item:basic-facts-on-Stone-dual-2}), the mapping $\varphi$
  is injective. Since, for $K,L\in\Cl B$, we have $\Cl U_K\cup\Cl
  U_L=\Cl U_{K\cup L}$ and $\Cl U_K^\complement=\Cl
  U_{K^\complement}$, it follows from
  (\ref{item:basic-facts-on-Stone-dual-1}) that $\varphi$~is a
  homomorphism of Boolean algebras. It remains to show that every
  clopen subset $C$ of~$\Cl B^\star$ belongs to the image
  of~$\varphi$. Now, as $C$ is open, it is a union of basic open sets
  $\Cl U_L$ with $L\in\Cl B$; as $C$ is compact, it is a finite union
  of such sets. Since $\bigcup_{i=1}^n\Cl U_{L_i}=\Cl
  U_{\bigcup_{i=1}^nL_i}$, we conclude that $C=\Cl U_L$ for some
  $L\in\Cl B$, as desired.
\end{proof}

\subsection{From Boolean algebras to Stone topological algebras}
\label{sec:Boole-to-Stone} \ 

The next result shows how to obtain Stone topological algebras from
the Stone dual of a Boolean algebra of subsets of~$S$. The algebraic
structure, for which suitable properties are stated and proved below,
is given by the following definition.

Let $S$ be a $\Omega$-algebra and let \Cl B be a Boolean subalgebra
of~$\Cl P(S)$. For $o\in\Omega_n$ and $u_1,\ldots,u_n\in\Cl B^\star$,
let
\begin{align}
  \label{eq:def-o-Bstar}
  \begin{split}
    \lefteqn{o_{\Cl B^\star}(u_1,\ldots,u_n)}\\
    &=\{L\in\Cl B: \forall i\in\{1,,\ldots,n\}\,\exists L_i\in u_i:\ 
      o_S(L_1\ldots,L_n)\subseteq L\}.
  \end{split}
\end{align}

\begin{Thm}
  \label{t:Boole-to-Stone}
  Let $\Omega$ be a topological signature, $S$ an $\Omega$-algebra,
  and \Cl B a Boolean subalgebra of~$\Cl P(S)$
  satisfying~(\ref{eq:star2}). Then, (\ref{eq:def-o-Bstar}) defines a
  structure of $\Omega$-algebra on $\Cl B^\star$ which makes it a
  Stone topological algebra. The mapping $\iota_{\Cl B}:S\to\Cl
  B^\star$ defined by sending each $s\in S$ to $s{\uparrow}=\{L\in\Cl
  B:s\in L\}$ is a homomorphism with dense image such that $\Cl
  B=\iota_{\Cl B}^{-1}(\Cl P_{co}(\Cl B^\star))$. Moreover, in case
  $S$ is a topological $\Omega$-algebra, the mapping $\iota_{\Cl B}$
  is continuous if and only if $\Cl B$ is contained in~$\Cl
  P_{co}(S)$.
\end{Thm}

\begin{proof}
  Given $o\in\Omega_n$ and $u_1,\ldots,u_n\in\Cl B^\star$, let
  $u=o_{\Cl B^\star}(u_1,\ldots,u_n)$ be the set defined
  by~(\ref{eq:def-o-Bstar}). We claim that $u$ is an ultrafilter
  of~\Cl B. It is immediate to check that it is a proper filter of~\Cl
  B. To show that it is an ultrafilter, we must show that, given
  $L\in\Cl B$, either $L$ or its complement $L^\complement$ belongs
  to~$u$. By the assumption that \Cl B satisfies~\eqref{eq:star2}, we
  know that both $o_S^{-1}(L)$ and $o_S^{-1}(L^\complement)$ belong to
  $\Cl B^{(n)}$: for instance, $o_S^{-1}(L)$ is the projection on the
  last $n$ components of $(E_n^S)^{-1}(L)\cap \{o\}\times S^n$. It
  follows that there are $n$ finite partitions
  $S=\biguplus_{j=1}^{k_i}L_{i,j}$ into elements of~\Cl B
  ($i=1,\ldots,n$) such that each box $L_{1,j_1}\times\cdots\times
  L_{n,j_n}$ is entirely contained in either $o_S^{-1}(L)$ or
  $o_S^{-1}(L^\complement)$. Since each $u_i$ is an ultrafilter of~\Cl
  B, there is a unique $\ell_i\in\{1,\ldots,k_i\}$ such that
  $L_{i,\ell_i}\in u_i$. For this choice of $\ell_i$, since the
  product $L_{1,\ell_1}\times\cdots\times L_{n,\ell_n}$ is contained
  in either $o_S^{-1}(L)$ or $o_S^{-1}(L^\complement)$, we conclude
  respectively that $L\in u$ or $L^\complement\in u$.
  
  Next, we claim that each evaluation mapping $E_n^{\Cl B^\star}$ is
  continuous. Let $L\in\Cl B$. By Lemma~\ref{l:Bn''-4}, there is a
  decomposition
  \begin{equation}
    \label{eq:Boole-to-Stone-1}
    (E_n^S)^{-1}(L)=\bigcup_{r\in R}K_r\times L_{1,r}\times
    \cdots\times L_{n,r}
  \end{equation}
  satisfying
  conditions~(C2.\ref{item:Bn''-1})--(C2.\ref{item:Bn''-4}). The claim
  follows from the formula
  \begin{equation}
    \label{eq:Boole-to-Stone-2}
    (E_n^{\Cl B^\star})^{-1}(\Cl U_L)
    =\bigcup_{r\in R}K_r\times\Cl U_{L_{1,r}}\times\cdots\times\Cl
    U_{L_{n,r}}
  \end{equation}
  which we proceed to establish. For the inclusion from right to left,
  suppose that $(o,u_1,\ldots,u_n)$ is such that there is $r\in R$
  with $o\in K_r$ and $L_{i,r}\in u_i$ ($i=1,\ldots,n$).
  From~(\ref{eq:Boole-to-Stone-1}), it follows that
  $o_S(L_{1,r},\ldots,L_{n,r})\subseteq L$, so that $L\in o_{\Cl
    B^\star}(u_1,\ldots,u_n)$, that is, $(o,u_1,\ldots,u_n)$ belongs
  to the left side of~(\ref{eq:Boole-to-Stone-2}).

  For the inclusion from left to right in~\eqref{eq:Boole-to-Stone-2},
  suppose that $(o,u_1,\ldots,u_n)$ belongs to $(E_n^{\Cl
    B^\star})^{-1}(\Cl U_L)$. Then, for each $i\in\{1,\ldots,n\}$,
  there exists $J_i\in u_i$ such that $o_S(J_1,\ldots,J_n)\subseteq
  L$, which means that $\{o\}\times J_1\times\cdots\times J_n\subseteq
  (E_n^S)^{-1}(L)$. Let $R_o=\{r\in R:o\in K_r\}$ and note that
  conditions~(C2.\ref{item:Bn''-1}) and (C2.\ref{item:Bn''-3})
  together imply that $R_o$ is a finite set. By
  (\ref{eq:Boole-to-Stone-1}), we deduce that
  $J_i\subseteq\bigcup_{r\in R_o} L_{i,r}$, so that this finite union
  also belongs to $u_i$; hence, there is $r_i$ such that $L_{i,r_i}\in
  u_i$ so that $J_i\cap L_{i,r_i}\in u_i$ and we may as well assume
  that $J_i$ is contained in some $L_{i,r_i}$. By
  Lemma~\ref{l:box-in-union-of-boxes}, it follows that there exists
  $r\in R$ such that $\{o\}\times J_1\times\cdots\times J_n\subseteq
  K_r\times L_{1,r}\times\cdots\times L_{n,r}$, thereby showing that
  $(o,u_1,\ldots,u_n)$ belongs to the right side
  of~(\ref{eq:Boole-to-Stone-2}).

  Next, we show that the correspondence $\iota_{\Cl B}:s\mapsto
  s{\uparrow}$ is a homomorphism $S\to\Cl B^\star$. Indeed, given
  $s_1,\ldots,s_n\in S$ and $o\in\Omega_n$, the ultrafilter $o_{\Cl
    B^\star}(s_1{\uparrow},\ldots,s_n{\uparrow})$ consists of all
  $L\in\Cl B$ such that, for each $i$, there exists $L_i\in
  s_i{\uparrow}$ such that $o_S(L_1,\ldots,L_n)\subseteq L$, so that,
  in particular, we have $o_S(s_1,\ldots,s_n)\in L$, that is, $L$
  belongs to $o_S(s_1,\ldots,s_n){\uparrow}$. Conversely, if
  $o_S(s_1,\ldots,s_n)\in L$, that is, $(o,s_1,\ldots,s_n)\in
  (E_n^S)^{-1}(L)$ then, since \Cl B is assumed to satisfy
  condition~(\ref{eq:star2}), there are $L_i\in\Cl B$ such that
  $s_i\in L_i$ ($i=1,\ldots,n$) and $\{o\}\times L_1\times\cdots\times
  L_n\subseteq (E_n^S)^{-1}(L)$, which implies that $L$ belongs to
  $o_{\Cl B^\star}(s_1{\uparrow},\ldots,s_n{\uparrow})$. We conclude
  that $o_{\Cl B^\star}(s_1{\uparrow},\ldots,s_n{\uparrow})
  =o_S(s_1,\ldots,s_n){\uparrow}$, thereby showing that $\iota_{\Cl
    B}$ is a homomorphism. Given a nonempty basic open set $\Cl U_L$,
  with $L\in\Cl B\setminus\{\emptyset\}$, we have $s{\uparrow}\in\Cl
  U_L$ for every $s\in L$, which shows that the image of~$\iota_{\Cl
    B}$ is dense in~$\Cl B^\star$.

  The equality $\iota_{\Cl B}^{-1}(\Cl P_{co}(\Cl B^\star))=\Cl B$
  follows from Proposition~\ref{p:basic-facts-on-Stone-dual}.

  It remains to deal with the continuity of~$\iota_{\Cl B}$. By
  Proposition~\ref{p:basic-facts-on-Stone-dual}, for every $L\in\Cl
  B$, we have $\iota_{\Cl B}^{-1}(\Cl U_L)=L$. Hence, $\iota_{\Cl B}$
  is continuous if and only if $\Cl B$ consists of open subsets
  of~$S$. Since $\Cl B$ is closed under complementation, we deduce
  that $\iota_{\Cl B}$ is continuous if and only if $\Cl B$ is
  contained in~$\Cl P_{co}(S)$.
\end{proof}

The next result shows how inclusion of Boolean subalgebras of~$\Cl
P(S)$ reflects on the corresponding Stone topological algebras.

\begin{Thm}
  \label{t:dual-of-inclusion}
  Let $\Omega$ be a topological signature, $S$ an $\Omega$-algebra,
  and \Cl B and \Cl C be Boolean subalgebras of~$\Cl P(S)$ satisfying
  condition~(\ref{eq:star2}) such that $\Cl B\subseteq\Cl C$. Consider
  on the dual spaces $\Cl B^\star$ and $\Cl C^\star$ the structure of
  Stone topological algebras given by Theorem~\ref{t:Boole-to-Stone}.
  Then the dual (surjective continuous) mapping $\xi^\star:\Cl
  C^\star\to\Cl B^\star$ of the inclusion $\xi:\Cl
  B\hookrightarrow\Cl C$ is a homomorphism of $\Omega$-algebras such
  that $\xi^\star\circ\iota_{\Cl C}=\iota_{\Cl B}$.
\end{Thm}

\begin{proof}
  We start by noting that, for an ultrafilter $u\in\Cl C^\star$, we
  have the equality $\xi^\star(u)=u\cap\Cl B$. To prove that
  $\xi^\star$ is a homomorphism of $\Omega$-algebras, consider an
  operation symbol $o\in\Omega_n$ and ultrafilters
  $u_1,\ldots,u_n\in\Cl C^\star$. From the definition of the
  interpretation of $o$ in the dual spaces $\Cl B^\star$ and $\Cl
  C^\star$, we see that
  \begin{align*}
    o_{\Cl B^\star}(u_1\cap\Cl B,\ldots,u_n\cap\Cl B)
    &=\{L\in\Cl B:\exists L_i\in u_i\cap\Cl B,\
      o_S(L_1,\ldots,L_n)\subseteq L\} \\
    &\subseteq\{L\in\Cl B:\exists L_i\in u_i,\
      o_S(L_1,\ldots,L_n)\subseteq L\} \\
    &=o_{\Cl C^\star}(u_1,\ldots,u_n)\cap\Cl B.
  \end{align*}
  Since both $o_{\Cl B^\star}(u_1\cap\Cl B,\ldots,u_n\cap\Cl B)$ and
  $o_{\Cl C^\star}(u_1,\ldots,u_n)\cap\Cl B$ are ultrafilters of~\Cl
  B, it follows that they are equal. Finally, for $s\in S$, both
  ultrafilters $\xi^\star(\iota_{\Cl C}(s))$ and $\iota_{\Cl B}(s)$
  are the set of all $L\in\Cl B$ such $s\in L$, that is, they are
  equal.
\end{proof}

\subsection{From Stone topological algebras to Boolean algebras}
\label{sec:Stone-to-Boole}

We next show how from a Stone topological algebra we may obtain a
Boolean subalgebra of~$\Cl P_{co}(S)$. This requires some preparation.

\begin{Lemma}
  \label{l:cont-homo-on-dense->homo}
  Let $\varphi:S\to T$ be a continuous mapping between two topological
  algebras and suppose that the restriction of~$\varphi$ to a dense
  subalgebra $A$ of~$S$ is a homomorphism. Then $\varphi$ is a
  homomorphism.
\end{Lemma}

\begin{proof}
  Let $s_1,\ldots,s_n$ be elements of $S$ and let $o\in\Omega_n$.
  Since $A$ is dense in $S$, for each $k\in\{1,\ldots,n\}$ there is a
  net $(a_{k,i})_i$ in $A$ converging to~$s_k$, where we may assume
  that the same index set is used for all~$k$. By continuity of the
  mappings $o_T$, $\varphi$, and $o_S$ together with the assumption
  that $\varphi|_A$ is a homomorphism, we obtain the following
  equalities:
  \begin{align*}
    o_T\bigl(\varphi(s_1),\ldots,\varphi(s_n)\bigr)
    &=\lim o_T\bigl(\varphi(a_{1,i}),\ldots,\varphi(a_{n,i})\bigr)\\
    &=\lim \varphi\bigl(o_S(a_{1,i},\ldots,a_{n,i})\bigr)\\
    &=\varphi\bigl(o_S(s_1,\ldots,s_n)\bigr).
  \end{align*}
  Hence, $\varphi$ is a homomorphism.
\end{proof}

\begin{Thm}
  \label{t:Stone-to-Boole}
  Suppose that $\varphi:S\to T$ is a continuous homomorphism between
  topological algebras whose image is dense in~$T$. Let $\Cl C$ be a
  Boolean subalgebra of $\Cl B_{\mathrm{max}2}^T$ satisfying
  condition~(\ref{eq:star2}) and let $\Cl B=\varphi^{-1}(\Cl C)$.
  Consider the mapping $\tilde{\varphi}:T\to\Cl B^\star$ defined
    by $\tilde{\varphi}(t)=\varphi^{-1}(t{\uparrow})$, where
    $t{\uparrow}=\{K\in\Cl C: t\in K\}$. Then $\tilde{\varphi}$ is a
    continuous homomorphism with dense image such that
    $\tilde{\varphi}\circ\varphi=\iota_{\Cl B}$. Moreover,
    $\tilde{\varphi}$ distinguishes two elements of~$T$ if and only if
    there is a member of~$\Cl C$ that separates them.
\end{Thm}

\begin{proof}
  By Proposition~\ref{p:lifting-star2}, the Boolean subalgebra \Cl B
  of~$\Cl P_{co}(S)$ satisfies~(\ref{eq:star2}). Note that the mapping
  $\psi:K\mapsto\varphi^{-1}(K)$ is an onto homomorphism of Boolean
  algebras $\Cl C\to\Cl B$. We claim that it is an isomorphism.
  Indeed, if $K$ and $L$ are distinct members of $\Cl C$ then the
  symmetric difference $K\vartriangle L$ is a nonempty open subset of
  $T$, whence it contains some element of the form $\varphi(s)$ for
  some $s\in S$, which implies that
  $\varphi^{-1}(K)\vartriangle\varphi^{-1}(L)$ is also nonempty,
  thereby showing that $\psi$ is injective. We thus get the external
  commuting square in the following diagram of continuous mappings,
  where the vertical arrows are homomorphisms with dense images given
  by Theorem~\ref{t:Boole-to-Stone}:
  \begin{displaymath}
    \xymatrix{
      S
      \ar[r]^\varphi
      \ar[d]_{\iota_{\Cl B}}
      &
      T
      \ar[d]^{\iota_{\Cl C}}
      \ar[ld]_{\tilde{\varphi}}
      \\
      \Cl B^\star
      \ar[r]_{\psi^\star}
      &
      \Cl C^\star
    }
  \end{displaymath}
  By Lemma~\ref{l:cont-homo-on-dense->homo}, we conclude that
  $\psi^\star$ is an isomorphism of Stone topological algebras.
  Further noting that $(\psi^\star)^{-1}\circ\iota_{\Cl
    C}=\tilde{\varphi}$, we deduce that $\tilde{\varphi}$ is indeed a
  continuous homomorphism with dense image and the equality
  $\tilde{\varphi}\circ\varphi=\iota_{\Cl B}$ follows from the
  commutativity of the outer square in the above diagram. The mapping
  $\tilde{\varphi}$ distinguishes two points of~$T$ if and only if so
  does $\iota_{\Cl C}$, which is equivalent to the condition that the
  points in question are separated by the members of~$\Cl C$.
\end{proof}

\begin{Cor}
  \label{c:Stone-to-Boole}
  Let $\varphi:S\to T$ be a continuous homomorphism with dense image,
  where $S$ is a topological algebra and $T$ is a Stone topological
  algebra. Then the Boolean algebra $\Cl B_\varphi=\{\varphi^{-1}(K):
  K\in\Cl P_{co}(T)\}$ satisfies~(\ref{eq:star2}) and $T$ is
  isomorphic with the Stone topological algebra $\Cl B_\varphi^\star$
  of Theorem~\ref{t:Boole-to-Stone}.
\end{Cor}

\begin{proof}
  By Propositions~\ref{p:Bmax2-in-Stone} and~\ref{p:lifting-star2},
  $\Cl B_\varphi$ satisfies~\eqref{eq:star2}. In the notation of the
  proof of Theorem~\ref{t:Stone-to-Boole}, all that remains to show is
  that $\iota_{\Cl C}$ is an isomorphism of Stone topological
  algebras, where $\Cl C=\Cl B_{\mathrm{max}2}^T$. In fact, we know
  that $\Cl C=\Cl P_{co}(T)$ by Proposition~\ref{p:Bmax2-in-Stone},
  hence $\iota_{\Cl C}$ is a homeomorphism. Since it is a homomorphism
  by Theorem~\ref{t:Boole-to-Stone}, it is indeed an isomorphism of
  Stone topological algebras.
\end{proof}

In Corollary~\ref{c:Stone-to-Boole}, one may take $S$ to be the term
algebra $T_\Omega(X)$ for a generating subset $X$ of~$T$. Then,
Corollary~\ref{c:Stone-to-Boole} shows that all Stone topological
algebras are duals of Boolean algebras of clopen subsets of a term
algebra satisfying~\eqref{eq:star2}. This may be viewed as an
alternative approach to duality compared with that adopted
in~\cite{Gehrke:2016a}.

\begin{Cor}
  \label{c:Bmax}
  Let $\varphi:T_\Omega(X)\to\Om X{\St}_\Omega$ be the natural
  homomorphism. Then the Boolean algebra $\Cl B_\varphi$ defined in
  Corollary~\ref{c:Stone-to-Boole} is the largest Boolean subalgebra
  of~$\Cl P_{co}\bigl(T_\Omega(X)\bigr)$ satisfying~(\ref{eq:star2}).
\end{Cor}

\begin{proof}
  By Propositions~\ref{p:Bmax2-in-Stone} and~\ref{p:lifting-star2},
  $\Cl B_\varphi$ is indeed a Boolean subalgebra of~$\Cl
  P_{co}\bigl(T_\Omega(X)\bigr)$ satisfying~\eqref{eq:star2}. By
  Theorem~\ref{t:Boole-to-Stone}, $(\Cl
  B_{\mathrm{max}2}^{T_\Omega(X)})^\star$ is an $X$-generated Stone
  topological algebra. The dual of the inclusion mapping $\eta:\Cl
  B_\varphi\to\Cl B_{\mathrm{max}2}^{T_\Omega(X)}$ is an onto
  continuous mapping $\eta^\star:(\Cl
  B_{\mathrm{max}2}^{T_\Omega(X)})^\star\to\Cl B_\varphi^\star$ which
  is a homomorphism respecting the generating mappings from the
  space~$X$ by Theorem~\ref{t:dual-of-inclusion}. But, by
  Corollary~\ref{c:Stone-to-Boole}, $\Cl B_\varphi^\star$ is freely
  generated by~$X$. Hence, $\eta^\star$ must be injective and, dually,
  $\eta$ must be surjective. This shows that $\Cl B_\varphi=\Cl
  B_{\mathrm{max}2}^{T_\Omega(X)}$.
\end{proof}

Note that, in view of Theorem~\ref{t:Boole-to-Stone} and
Corollary~\ref{c:Stone-to-Boole}, Corollary~\ref{c:Bmax} may be
thought as providing a construction of the absolutely free Stone
topological algebra $\Om X{\St}_\Omega$, modulo the identification of
the Boolean algebra $\Cl B_{\mathrm{max}2}^{T_\Omega(X)}$, for which we have no
constructive description.

\begin{Thm}
  \label{t:Bmax-proper}
  Let $\Omega$ be a topological signature for which there is $n>1$
  such that $\Omega_n$ and $X$ are both nonempty discrete spaces. Then
  $\Cl B_{\mathrm{max}2}^{T_\Omega(X)}$ is a proper Boolean subalgebra
  of~$\Cl P_{co}\bigl(T_\Omega(X)\bigr)$.
\end{Thm}

\begin{proof}
 Fix $n>1$ such that $\Omega_n$ is nonempty and discrete and consider
 the signature $\Omega'=\Omega'_n=\Omega_n$. Notice, that $T_{\Omega'}(X)$ 
 is a closed subset of~$T_\Omega(X)$, and for each $t\in T_{\Omega'}(X)$, 
 the set $\{t\}$ is an open subset of~$T_\Omega(X)$. Thus, every subset of
 $T_{\Omega'}(X)$ is clopen in~$T_\Omega(X)$. 
 Let $L$ consist of all elements of $T_{\Omega'}(X)$ of the form
 $o(t,\ldots,t)$ with $o\in\Omega_n$ and $t\in T_{\Omega'}(X)$. For
 any $o\in\Omega_n$, the set
 $E_n^{-1}(L)\cap\{o\}\times(T_\Omega(X))^n$ is not a finite union of
 boxes, while $L$ is a clopen subset of~$T_\Omega(X)$. Hence, the
 Boolean algebra %
 $\Cl P_{co}\bigl(T_\Omega(X)\bigr)$ does not
 satisfy~\eqref{eq:star2}, so that $\Cl
 B_{\mathrm{max}2}^{T_\Omega(X)}$ must be a proper Boolean subalgebra
 of~$\Cl P_{co}\bigl(T_\Omega(X)\bigr)$.
\end{proof}

We may now derive a result showing that, at least beyond the unary
case, the \v Cech-Stone compactification does not provide a
construction for the free Stone topological algebra~$\Om
X{\St}_\Omega$.

\begin{Cor}
  \label{c:non-unary}
  Let $\Omega$ be a discrete signature with at least one operation
  symbol of arity greater than 1 and let $X$ be a nonempty discrete
  space. Then it is not possible to define on the Stone space
  $\beta\bigl(T_ \Omega(X)\bigr)$ a structure of topological algebra
  in which $T_\Omega(X)$ is a subalgebra.
\end{Cor}

\begin{proof}
  Suppose on the contrary that the Stone space %
  $S=\beta\bigl(T_ \Omega(X)\bigr)$ admits a structure of topological
  algebra extending the structure of $T_\Omega(X)$ and let
  $\varphi:T_\Omega(X)\to S$ %
  be the inclusion mapping. By Theorem~\ref{t:Stone-to-Boole}, we have
  an associated Boolean subalgebra $\Cl B_\varphi=\{\varphi^{-1}(K):
  K\in\Cl P_{co}(S)\}$ %
  of~$\Cl P\bigl(T_\Omega(X)\bigr)$ satisfying~\eqref{eq:star2}. Since
  $T_\Omega(X)$ is a discrete space, $\Cl B_\varphi$ is equal to $\Cl
  P\bigl(T_\Omega(X)\bigr)$ (see, for instance,
  \cite[Theorem~3.27]{Hindman&Strauss:2012;bk}). But, since $\Cl
  B_\varphi$ satisfies~\eqref{eq:star2}, it must be contained in~$\Cl
  B_{\mathrm{max}2}^{T_\Omega(X)}$. It follows that $\Cl
  B_{\mathrm{max}2}^{T_\Omega(X)}=\Cl P\bigl(T_\Omega(X)\bigr)$, which
  contradicts Theorem~\ref{t:Bmax-proper}.
\end{proof}

The special case of Theorem~\ref{t:Stone-to-Boole} for a profinite
algebra $T$ is particularly interesting. Suppose that $T$ is generated
by a continuous mapping $\varphi:X\to T$ and consider the unique
continuous homomorphic extension $\hat{\varphi}:T_\Omega(X)\to T$.
Recall that $\Cl B_{\hat{\varphi}}$ consists of all subsets
of~$T_\Omega(X)$ of the form $\hat{\varphi}^{-1}(K)$ where $K$ is a
clopen subset of~$T$. Now, by Theorem~\ref{t:profiniteness}, a subset
$K$ of~$T$ 
is clopen if and only if there is a continuous homomorphism
$\psi:T\to F$ onto a finite algebra $F$ such that
$K=\psi^{-1}(\psi(K))$. Hence, $\Cl B_{\hat{\varphi}}$ consists of the
sets of the form $L=\hat{\varphi}^{-1}\bigl(\psi^{-1}(P)\bigr)$ with
$P$ an arbitrary subset of~$F$. In case $T=\Om XV$ is a relatively
free profinite algebra, over a pseudovariety \pv V of finite algebras,
the composite homomorphisms $\psi\circ\hat{\varphi}$ may be
characterized simply as the onto continuous homomorphisms
$T_\Omega(X)\to F$ with $F\in\pv V$. This leads to the notion of
\emph{\pv V-recognizable tree language} (better known as \emph{regular
  tree language} \cite{Gecseg&Steinby:1997,Steinby:2005} in case $\pv
V=\pv{Fin}_\Omega$) of computer science and so $\Cl B_{\hat{\varphi}}$
consists precisely of all such languages over the ``alphabet'' $X$.
Characterizations of the Boolean algebras $\Cl B_{\hat{\varphi}}$ with
$X$ finite and discrete can be found as part of the analog of
Eilenberg's Correspondence Theorem (or ``Variety Theorem'')
\cite{Almeida:1994a,Steinby:1992,Steinby:2005}.

\section{Stone varieties}
\label{sec:Stone-varieties}

In this section, we present one of our main results and applications.
It is a generalization for topological signatures of
Theorem~\ref{t:quotient-profinite}, that result being the special case
of the Stone pseudovariety $\pv{Fin}_\Omega$.

\begin{Thm}
  \label{t:residual}
  Let $\Omega$ be an arbitrary topological signature and let \Cl S be
  an arbitrary Stone pseudovariety. Let $\varphi:S\to T$ be an onto
  continuous homomorphism of Stone topological algebras where $S$ is
  residually \Cl S. Then $T$ is also residually~\Cl S.
\end{Thm}

\begin{proof}
  Let $t_1,t_2$ be a pair of distinct points of~$T$. Since $T$ is a
  Stone space, there is a clopen subset $K$ of~$T$ such that $t_1\in
  K$ and $t_2\notin K$. Applying Lemma~\ref{l:recognition} to the
  clopen subset $L=\varphi^{-1}(K)$ of~$S$, we obtain a continuous
  homomorphism $\psi:S\to U$ onto a member of~\Cl S such that
  $L=\psi^{-1}(\psi(L))$.
  Then, in the notation of Corollary~\ref{c:Stone-to-Boole}, by
  Theorem~\ref{t:Stone-to-Boole}, the Boolean subalgebras $\Cl
  B_\varphi$ and $\Cl B_\psi$ of $\Cl P_{co}(S)$ both
  satisfy~\eqref{eq:star2}. By Corollary~\ref{c:Bcap}, the Boolean
  algebra $\Cl B=\Cl B_\varphi\cap\Cl B_\psi$ also
  satisfies~\eqref{eq:star2}. Consider the various inclusions between
  the four Boolean algebras $\Cl B$, $\Cl B_\varphi$, $\Cl B_\psi$,
  and $\Cl P_{co}(S)$, where we denote $\Cl B\hookrightarrow\Cl
  B_\varphi$ by $\xi$. Dualizing, in view of
  Theorem~\ref{t:dual-of-inclusion}, we obtain the following
  commutative diagram of continuous homomorphisms between Stone
  topological algebras, where $\tilde{\varphi}$ 
  is given by
  Theorem~\ref{t:Stone-to-Boole} and all mappings are onto:
  \begin{displaymath}
    \xymatrix@C=10mm{
      S
      \ar[r]^\varphi
      \ar[d]_\psi
      \ar@/^20pt/[rr]^{\iota_{\Cl B_\varphi}}
      \ar@/_/[rrd]^{\iota_{\Cl B}}
      &
      T
      \ar[rd]^\delta
      \ar[r]^(.4){\tilde{\varphi}}
      &
      (\Cl B_\varphi)^\star
      \ar[d]^{\xi^\star}
      \\
      U \ar[rr]^\varepsilon
      &&
      \Cl B^\star
    }
  \end{displaymath}
  In particular, since $U$ belongs to the
  Stone pseudovariety \Cl S, so does the Stone topological algebra
  $\Cl B^\star$. To finish the proof, it suffices to show that
  $\delta(t_1)\ne\delta(t_2)$.

  Note that $L\in\Cl B_\varphi\cap\Cl B_\psi=\Cl B$. Let $s_i\in S$ be
  such that $\varphi(s_i)=t_i$ ($i=1,2$). Then we have $s_1\in L$
  while $s_2\notin L$. As $\delta\circ\varphi=\iota_{\Cl B}$ and, for
  each $s\in S$, we have $\iota_{\Cl B}(s)=\{J\in\Cl B: s\in J\}$, we
  get $\delta(t_1)=\iota_{\Cl B}(s_1)\ne\iota_{\Cl
    B}(s_2)=\delta(t_2)$, thereby completing the proof of the theorem.
\end{proof}

For a class \Cl C of Stone topological algebras, we denote by
$\widehat{\Cl C}$ the class of all residually \Cl C Stone topological
algebras. Note that $\widehat{\widehat{\Cl C}}=\widehat{\Cl C}$. Thus,
the correspondence $\Cl C\mapsto\widehat{\Cl C}$ is a closure operator
on the class of all classes of Stone topological algebras. We call
$\widehat{\Cl C}$ the \emph{residual closure} of~\Cl C. We say that a
class \Cl C of Stone topological algebras \Cl S is \emph{residually
  closed} if $\Cl C=\widehat{\Cl C}$.

\begin{Cor}
  \label{c:residual-closure}
  Let $\Omega$ be a topological signature and \Cl S a Stone
  pseudovariety. Then $\widehat{\Cl S}$ is also a Stone pseudovariety
  and, for every topological space $X$, we have $\Om X{\Cl S}\simeq\Om
  X{\widehat{\Cl S}}$.
\end{Cor}

\begin{proof}
  To check that $\widehat{\Cl S}$ is a Stone pseudovariety, the only
  nontrivial requirement is that it be closed under taking Stone
  continuous homomorphic images, which follows from
  Theorem~\ref{t:residual}. The existence of a continuous
  homomorphism %
  $\varphi:\Om X{\Cl S}\to\Om X{\widehat{\Cl S}}$ %
  respecting generating mappings is an immediate consequence of
  Proposition~\ref{p:free-Stone2}. The existence of a continuous
  homomorphism %
  $\Om X{\widehat{\Cl S}}\to\Om X{\Cl S}$ %
  also respecting generating mappings follows from the obvious fact
  that $\psi:\Cl S\subseteq\widehat{\Cl S}$. To conclude the proof, it
  suffices to observe that $\varphi$ and $\psi$ are mutualy inverse
  mappings.
\end{proof}

By a \emph{Stone variety} we mean a nonempty class of Stone
topological algebras that is closed under taking continuous
homomorphic images that are again Stone spaces, closed subalgebras,
and arbitrary direct products.

\begin{Prop}
  \label{p:varieties}
  A class of Stone topological algebras is a Stone variety if
  and only if it is a residually closed Stone pseudovariety.
\end{Prop}

\begin{proof}
  A Stone variety \Cl V is obviously a Stone pseudovariety. It is
  residually closed because, if $S$ is residually~\Cl V, then there is
  an embedding of $S$, as a closed subalgebra, into a product of
  members of~\Cl V, and so $S$ belongs to~\Cl V.

  Conversely, we claim that, if \Cl S is a residually closed Stone
  pseudovariety, then \Cl S is a Stone variety. Let $(S_i)_{i\in I}$
  be a nonempty family of members of~\Cl S. Given two distinct
  elements of the product $S=\prod_{i\in I}S_i$, they are
  distinguished by the projection on some component $S_i$, and so $S$
  is residually~\Cl S. Since \Cl S is residually closed, we deduce
  that $S$ belongs to~\Cl S, which proves the claim.
\end{proof}

Note that, for a Stone variety \Cl V and an arbitrary topological
space~$X$, the relatively free Stone topological algebra %
$\Om X{\Cl V}$ belongs to~\Cl V, a fact that follows from the
construction of $\Om X{\Cl V}$ in the proof of Proposition
\ref{p:free-Stone}. We deduce that a Stone topological algebra belongs
to~\Cl V if and only if it is a continuous homomorphic image of some
$\Om X{\Cl V}$.

An immediate consequence of Proposition~\ref{p:varieties} is that
the residual closure of a Stone pseudovariety is the Stone variety it
generates.

The proof of Theorem~\ref{t:residual} may be adapted to establish the
following result.

\begin{Thm}
  \label{t:bigcap}
  If $(\Cl S_i)_{i\in I}$ is a nonempty family of Stone
  pseudovarieties then $\reallywidehat{\bigcap_{i\in I}\Cl
    S_i}=\bigcap_{i\in I}\widehat{\Cl S_i}$.
\end{Thm}

\begin{proof}
  Since the inclusion %
  $\reallywidehat{\bigcap_{i\in I}\Cl S_i}\subseteq \bigcap_{i\in
    I}\widehat{\Cl S_i}$ is clear, we need to show that, if $T$ is a
  Stone topological algebra that is residually $\Cl S_i$ for each
  $i\in I$, then $T$ is also residually $\bigcap_{i\in I}\Cl S_i$. Let
  $t_1$ and $t_2$ be distinct elements of~$T$ and choose a clopen
  subset $K$ of~$T$ such that $t_1\in K$ and $t_2\notin K$. By
  Lemma~\ref{l:recognition}, there are onto continuous homomorphisms
  $\varphi_i:T\to S_i$ such that $S_i\in\Cl S_i$ and
  $K=\varphi_i^{-1}(\varphi_i(K))$ ($i\in I$). Since every mapping
  $\varphi_i$ is closed, the image $\varphi_i(K)$ is a clopen subset
  of $S_i$.

  Consider the Boolean subalgebras $\Cl
  B_{\varphi_i}=\{\varphi_i^{-1}(L): L\in\Cl P_{co}(S_i)\}$ of $\Cl
  P_{co}(T)$, which satisfies condition~\eqref{eq:star2} by
  Corollary~\ref{c:Stone-to-Boole}. By Corollary~\ref{c:Bcap2}, their
  intersection $\Cl B$ is also a Boolean subalgebra of~$\Cl P_{co}(T)$
  satisfying \eqref{eq:star2}. We get the following diagram of
  inclusions between Boolean algebras:
  \begin{displaymath}
    \xymatrix{
      \Cl P_{co}(T)
      &
      \Cl B_{\varphi_i}
      \ar@{^{((}->}[l]
      \\
      &
      **[r]{\Cl B.}
      \ar@{^{((}->}[lu]
      \ar@{^{((}->}[u]
    }
  \end{displaymath}
  The dual diagram of continuous homomorphisms yields the lower
  triangle of the following commutative diagram for every $i\in I$:
  \begin{displaymath}
    \xymatrix@C=12mm{
      &
      S_i
      \ar[d]^{\tilde{\varphi}_i}
      \\
      T \ar[r]^{\iota_{\Cl B{\varphi_i}}}
      \ar[rd]_{\iota_{\Cl B}}
      \ar[ru]^{\varphi_i}
      & \Cl B_{\varphi_i}^\star
      \ar[d]^{\psi_i} \\
      & **[r]{\Cl B^\star.}
    }
  \end{displaymath}
  By the assumptions on the $S_i$, we see that $K\in\Cl B_{\varphi_i}$
  for every $i\in I$, whence $K\in\Cl B$. By
  Theorem~\ref{t:Stone-to-Boole}, we get $\iota_{\Cl
    B}(t_1)\ne\iota_{\Cl B}(t_2)$. On the other hand, the continuous
  homomorphisms $\psi_i\circ\tilde{\varphi}_i$ are onto. Since each
  $\Cl S_i$ is a Stone pseudovariety, it follows that $\Cl B^\star$
  belongs to $\bigcap_{i\in I}\Cl S_i$. Hence, $T$ is residually
  $\bigcap_{i\in I}\Cl S_i$.
\end{proof}

Thus, the residual closure operator on Stone pseudovarieties is a
complete meet endomorphism of the lattice $\cal L_\Omega$ of all Stone
pseudovarieties.

If \Cl S is a residually closed Stone pseudovariety, then we may
consider the family of all Stone pseudovarieties $\Cl S_i$ such that
$\widehat{\Cl S_i}=\Cl S$. By Theorem~\ref{t:bigcap}, we conclude that
$\reallywidehat{\bigcap_{i\in I}\Cl S_i}=\Cl S$. This proves the
following result.

\begin{Cor}
  \label{c:interval}
  Given a Stone variety \Cl V, the set of all Stone pseudovarieties
  \Cl S with $\widehat{\Cl S}=\Cl V$ is an interval $[\Cl V_{\min},\Cl
  V]$ of the lattice $\Cl L_\Omega$.\qed
\end{Cor}

In the profinite case, the minimum of the interval of
Corollary~\ref{c:interval} admits a particularly simple description.

\begin{Thm}
  \label{t:profinite}
  If \Cl V is a Stone variety of profinite algebras then $\Cl
  V_{\min}=\Cl V\cap\pv{Fin}_\Omega$.
\end{Thm}

\begin{proof}
  From $\Cl V\cap\pv{Fin}_\Omega\subseteq\Cl V$, we deduce that
  $\reallywidehat{\Cl V\cap\pv{Fin}_\Omega}\subseteq\widehat{\Cl
    V}=\Cl V$. On the other hand, if~$S$ belongs to~\Cl V, then $S$ is
  profinite and so $S$ embeds in a product of its finite quotients,
  thus in a product of members of~$\Cl V\cap\pv{Fin}_\Omega$, which
  proves that %
  $\Cl V\subseteq\reallywidehat{\Cl V\cap\pv{Fin}_\Omega}$. Hence,
  $\Cl V\cap\pv{Fin}_\Omega$ is a pseudovariety of finite algebras
  whose residual closure is~\Cl V. To complete the proof, we claim
  that, if \pv V and \pv W are pseudovarieties of finite algebras such
  that $\pv V\subseteq\pv W\subseteq\widehat{\pv V}$, then $\pv V=\pv
  W$. Indeed, if $S$ is an arbitrary element of~\pv W, then $S$ is
  residually~\pv V, that is, it embeds in a direct product of algebras
  from~\pv V. Since $S$ is finite, it suffices to consider only
  finitely many factors in such a product to achieve the embedding.
  Hence, $S$ belongs to~\pv V.
\end{proof}

\begin{eg}
  \label{eg:profinite}
  Let $\Omega$ be a 0-dimensional signature. By
  Corollary~\ref{c:residual-closure} and
  Proposition~\ref{p:varieties}, the class
  $\reallywidehat{\pv{Fin}_\Omega}$ of all profinite algebras is a
  Stone variety. Suppose also that the space $X$ is 0-dimensional. By
  Proposition~\ref{p:free-Stone-embedding-of-terms}, the topological
  algebra $T_\Omega(X)$ is residually finite. Hence, the variety of
  all profinite $\Omega$-algebras satisfies no nontrivial identities.
  Of course, neither does the variety $\St_\Omega$ of all Stone
  topological $\Omega$-algebras, which contains non-profinite algebras
  by Theorem~\ref{t:Stone-free-not-profinite}. This shows that one
  cannot expect Stone varieties to be defined by identities as in the
  classical Birkhoff theorem for varieties of discrete algebras over
  discrete signatures \cite{Birkhoff:1935}. Nevertheless, of course
  every set $\Sigma$ of identities still defines a Stone variety,
  namely the class of all Stone topological algebras that satisfy
  $\Sigma$.
\end{eg}

A topological space $X$ is said to be \emph{extremally disconnected}
if the closure of every open subset of~$X$ is open. It is easy to see
that every Hausdorff extremally disconnected space is totally
disconnected. The compact extremally disconnected spaces are sometimes
called \emph{Stonean spaces} and turn out to be, up to homeomorphism,
the Stone duals of complete Boolean algebras. Complete Boolean
algebras are known to be injective in the category of Boolean algebras
\cite{Sikorski:1948a} (that is, given a Boolean subalgebra $A$ of a
Boolean algebra $B$, a homomorphism from $A$ to a complete
Boolean algebra $C$ extends to a homomorphism $B\to C$). Dually, and
more generally, Stonean spaces are known to be precisely the
projective spaces in the category of compact spaces
\cite{Gleason:1958}. For more details, see \cite{Johnstone:1986}. Note
that, since a Boolean algebra may always be embedded in a complete
Boolean algebra, dually, every Stone space is a continuous image of
some Stonean space, a fact that is used below.

In spite of Example~\ref{eg:profinite}, it is still natural to expect
Stone varieties to be defined by some sort of identities and it turns
out that Stone pseudoidentities, which we proceed to introduce, play
that role.

Let $X$ be a topological space and let \Cl S be a Stone pseudovariety.
Let $\iota:X\to\Om X{\Cl S}$ be the natural generating mapping. We
associate with each continuous mapping $\varphi:X\to S$ into a Stone
topological algebra $S$ which is residually~\Cl S the unique
continuous homomorphism %
$\hat{\varphi}:\Om X{\Cl S}\to S$ such that
$\hat{\varphi}\circ\iota=\varphi$. By a \emph{Stone \Cl
  S-pseudoidentity over $X$} we mean a pair $(u,v)$ of elements %
of~$\Om X{\Cl S}$, usually written as a formal equality $u=v$. We say
that a residually~\Cl S Stone topological algebra $S$ \emph{satisfies}
$u=v$ if, for every continuous homomorphism $\varphi:X\to S$, the
equality $\hat{\varphi}(u)=\hat{\varphi}(v)$ holds. In case $\Cl
S=\St_\Omega$ consists of all Stone topological $\Omega$-algebras, we
refer simply to \emph{Stone pseudoidentities over~$X$}.

The following result is the suitable analog of Birkhoff's theorem for
Stone varieties. It comes at the cost of allowing proper classes of
Stone pseudoidentities for describing varieties for, unlike the
classical discrete setting, one cannot reduce to the case where
(pseudo)identities are written over finite sets of variables. More
precisely, we consider classes consisting of Stone \Cl
S-pseudoidentities over arbitrary Stonean spaces~$X$. For such a class
$\Sigma$, we denote $\op\Sigma\cl_{\Cl S}$ the class of all members
of~\Cl S that satisfy all members of~$\Sigma$ and we call $\Sigma$ a
\emph{basis} of the class~$\op\Sigma\cl_{\Cl S}$. In case $\Cl
S=\St_\Omega$, we drop the index \Cl S, writing simply $\op\Sigma\cl$.

\begin{Thm}
  \label{t:Birkhoff}
  A class \Cl V of Stone topological algebras is a Stone variety if
  and only if there is a class $\Sigma$ of Stone pseudoidentities over
  Stonean spaces such that $\Cl V=\op\Sigma\cl$.
\end{Thm}

\begin{proof}
  Suppose first that \Cl V is a Stone variety: let $\Sigma$ be the
  class consisting of all Stone pseudoidentities over Stonean spaces
  that are valid in~\Cl V. We claim that %
  $\op\Sigma\cl\subseteq\Cl V$, the reverse inclusion being obvious
  from the choice of~$\Sigma$. Let $S$ be an arbitrary element
  of~$\op\Sigma\cl$ and consider a generating mapping %
  $\varphi:X\to S$ with $X$ a Stone space, where we could simply take
  the identity mapping on~$S$. Since $X$ is a continuous image of some
  Stonean space, we may as well assume that $X$ is a Stonean space.
  Then, with the above choice of~$\iota$ and~$\hat{\varphi}$, we may
  also consider the natural continuous homomorphism $\eta$ in
  Diagram~(\ref{eq:Reiterman}) (cf.~end of
  Subsection~\ref{sec:free-Stone}).
  \begin{equation}
    \label{eq:Reiterman}
    \begin{gathered}
      \xymatrix{ X \ar[r]^(.4)\iota \ar[rd]_\varphi & \Om
        X{\St}_\Omega \ar[d]^{\hat{\varphi}} \ar[r]^\eta
        & \Om X{\Cl V} \ar@{-->}[ld]^\psi \\
        &S & }
    \end{gathered}
  \end{equation}
  Note that, if the pair of elements $u,v\in\Om X{\St}_\Omega$ is such
  that $\eta(u)=\eta(v)$, then from the universal property of~$\Om
  X{\Cl V}$ we deduce that the Stone pseudoidentity $u=v$ belongs
  to~$\Sigma$. By the assumption that $S$ belongs to $\op\Sigma\cl$,
  it follows that $\hat{\varphi}(u)=\hat{\varphi}(v)$ for all such
  pairs $u,v$. Hence, there is a unique homomorphism $\psi$ such that
  Diagram~\eqref{eq:Reiterman} commutes and continuity of $\psi$
  follows from the continuity of $\hat{\varphi}$ and $\eta$ and
  compactness of~$\Om X{\St}_\Omega$. Since $\varphi$ is a generating
  mapping, $\psi$ is onto. Since \Cl V is a Stone variety and $\Om
  X{\Cl V}$ belongs to~\Cl V, we deduce that $S\in\Cl V$. This
  establishes the equality %
  $\Cl V=\op\Sigma\cl$ and proves half of the theorem.

  The proof of the converse consists in showing that %
  $\Cl V=\op\Sigma\cl$ is a Stone variety for every class $\Sigma$ of
  Stone pseudoidentities over Stonean spaces. Thus, we should show
  that \Cl V is closed under taking quotients that are Stone spaces,
  closed subalgebras, and arbitrary direct products. Except for the
  case of quotients, the verification is standard and amounts to a
  straightforward argument that is omitted. The exception is what
  leads us to consider Stone pseudoidentities over Stonean spaces
  rather than over arbitrary topological spaces.

  So, let $\alpha:S\to T$ be an arbitrary onto continuous homomorphism
  between Stone topological algebras and assume that $S\in\Cl V$. Let
  $u=v$ be a member of~$\Sigma$, say $u,v\in\Om X{\St}_\Omega$ for a
  Stonean space~$X$. Let $\varphi:X\to T$ be a continuous mapping.
  Since $X$ is projective in the category of Stone spaces (cf.\ above
  discussion), there is a continuous mapping $\psi:X\to S$ such that
  $\alpha\circ\psi=\varphi$. Consider the induced continuous
  homomorphisms $\hat{\varphi}$ and $\hat{\psi}$ such that the
  following diagram commutes:
  
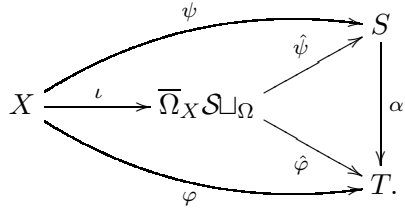
\begin{displaymath}
    \xymatrix@R=5mm@C=8mm{
      &&& S \ar[dd]^\alpha \\
      X \ar@/^5mm/[rrru]^\psi \ar[rr]^(.35)\iota \ar@/_5mm/[rrrd]_\varphi
      && **[l]{\Om X{\St}_\Omega} \ar[ru]^{\hat{\psi}} \ar[rd]_{\hat{\varphi}}
      & \\
      &&& **[r]{T.}
    }
  \end{displaymath}
  Since $S$ belongs to~\Cl V, $S$ satisfies $u=v$ and so the equality
  $\hat{\psi}(u)=\hat{\psi}(v)$ holds. From the commutativity of the
  diagram, it follows that $\hat{\varphi}(u)=\hat{\varphi}(v)$, which
  shows that $T$ also satisfies $u=v$. This establishes that $T\in\Cl
  V$ and completes the proof of the theorem.
\end{proof}

In particular, the Stone variety $\reallywidehat{\pv{Fin}_\Omega}$ of
all profinite $\Omega$-algebras is defined by some class of Stone
pseudoidentities over Stonean spaces. Other than the basis provided by
the proof of Theorem~\ref{t:Birkhoff}, we know of no simple basis
for~$\reallywidehat{\pv{Fin}_\Omega}$. Yet, for some pseudovarieties
\pv V of finite algebras, usual identities (between terms) are
sufficient to define the Stone variety $\widehat{\pv V}$ of all
pro-\pv V algebras. Semigroups, groups, rings distributive lattices
and lattices \cite{Numakura:1957,Clark&Davey&Freese&Jackson:2004} are
examples where this phenomenon occurs.

We may derive Reiterman's theorem from our previous results. Although
the proof is not shorter than a direct proof, it shows how Reiterman's
theorem can be viewed as a special case of Theorem~\ref{t:Birkhoff}.

\begin{Thm}
  \label{t:Reiterman}
  A class of finite topological algebras is a Stone pseudovariety if
  and only if it is of the form $\op\Sigma\cl_{\pv{Fin}_\Omega}$ for
  some set $\Sigma$ of Stone $\pv{Fin}_\Omega$-pseudoidentities over
  finite discrete spaces.
\end{Thm}

\begin{proof}
  Let \Cl S be a Stone pseudovariety and $X$ a Stonean space. For
  $u,v\in\Om X{\Cl S}$, we may choose $u',v'\in\Om X{\St}_\Omega$ such
  that $\eta(u')=u$ and $\eta(v')=v$, where $\eta:\Om
  X{\St}_\Omega\to\Om X{\Cl S}$ is the natural continuous
  homomorphism. We claim that, for an arbitrary Stone topological
  algebra $S$ which is residually~\Cl S, $S$~satisfies $u=v$ if and
  only if $S$~satisfies $u'=v'$. Indeed, if $S$ satisfies $u=v$ and
  $\varphi:\Om X{\St}_\Omega\to S$ is a continuous homomorphism, then
  there is a continuous homomorphism $\psi:\Om X{\Cl S}\to S$
  such that $\psi\circ\eta=\varphi$. By assumption,
  $\varphi(u)=\varphi(v)$ and so we get that $\psi(u')=\psi(v')$. The
  converse is even simpler and is left to the reader. Applying the
  claim to the Stone pseudovariety $\Cl S=\pv{Fin}_\Omega$, it follows
  that, for $\Sigma$ as in the statement of the theorem, %
  \begin{displaymath}
    \op\Sigma\cl_{\pv{Fin}_\Omega}
    =\op u'=v':(u=v)\in\Sigma\cl \cap \pv{Fin}_\Omega
  \end{displaymath}
  is the intersection of two
  Stone pseudovarieties by Theorem~\ref{t:Birkhoff}, whence it is also
  a Stone pseudovariety.

  For the converse, let \pv V be a pseudovariety of finite topological
  algebras. By Theorem~\ref{t:Birkhoff}, there is a set $\Sigma$ of
  $\St_\Omega$-pseudoidentities over Stonean spaces such that
  $\widehat{\pv V}=\op\Sigma\cl$. For each member $u=v$ of~$\Sigma$,
  say with $u,v\in\Om X{\St}_\Omega$, we may consider the
  $\widehat{\pv V}$-pseudoidentity $\eta(u)=\eta(v)$, where $\eta:\Om
  X{\St}_\Omega\to\Om X{\widehat{\pv V}}$ is the natural continuous
  homomorphism. Let $\Sigma'$ be the set of all such $\widehat{\pv
    V}$-pseudoidentities. By the considerations at the beginning of
  the proof and Theorem~\ref{t:profinite}, we obtain the equalities
  \begin{displaymath}
    \op\Sigma'\cl_{\pv{Fin}_\Omega}
    =\op\Sigma\cl \cap \pv{Fin}_\Omega
    =\widehat{\pv V} \cap \pv{Fin}_\Omega
    =\pv V.
  \end{displaymath}
  It remains to show that we only need to take
  $\pv{Fin}_\Omega$-pseudoidentities over finite discrete spaces.
  Given a $\pv{Fin}_\Omega$-pseudoidentity $u=v$, we consider all
  pseudoidentities of the form $\gamma(u)=\gamma(v)$ where $\gamma:\Om
  X{Fin}_\Omega\to\Om Y{Fin}_\Omega$ is an arbitrary continuous
  homomorphism and $Y$ is an arbitrary finite subspace of a fixed
  countable discrete space~$Z$. A routine argument shows that a
  profinite algebra satisfies $u=v$ if and only if it satisfies all
  such pseudoidentities over finite discrete spaces. Thus, if we
  consider all pseudoidentities over finite subspaces of~$Z$ so
  associated with pseudoidentities from~$\Sigma'$, we obtain a set of
  pseudoidentities $\tilde{\Sigma}$ such that
  $\op\tilde{\Sigma}\cl_{\pv{Fin}_\Omega}=\pv V$.
\end{proof}

We finish with an example showing that Stone varieties are not
characterized by their profinite members.

\begin{eg}
  \label{eg:Jonsson}
  Let $\Omega=\Omega_1\cup\Omega_2$ be the signature given by
  $\Omega_1=\{\alpha,\beta\}$ and $\Omega_2=\{\gamma\}$. Consider the
  Stone variety \Cl V defined by the J\'onsson-Tarski identities
  \cite{Jonsson&Tarski:1961}:
  \begin{equation}
    \label{eq:J-T-identities}
    \alpha\bigl(\gamma(x,y)\bigr)=x,\
    \beta\bigl(\gamma(x,y)\bigr)=y,\
    \gamma\bigl(\alpha(x),\beta(x)\bigr)=x.
  \end{equation}
  Note that an $\Omega$-algebra $S$ satisfies the
  identities~\eqref{eq:J-T-identities} if and only if the mappings
  $\gamma_S:S\times S\to S$ and $(\alpha_S,\beta_S):S\to S\times S$
  are mutual inverses. In particular, the only finite algebra in~\Cl V
  is the trivial one. Since \Cl V is a Stone variety, it follows that
  it does not contain any nontrivial profinite algebra. Note also that
  \Cl V is nontrivial as the Cantor set $C$ admits a homeomorphism
  $C\times C\to C$ and so it has a structure that makes it a member
  of~\Cl V.
\end{eg}

\bibliographystyle{amsplain}
\bibliography{sgpabb,ref-sgps}

\end{document}